\documentclass[reqno]{amsart}
 \usepackage{amssymb,amsmath,amscd,graphicx,color,epstopdf,mathtools,comment}
\usepackage[all]{xy}
\usepackage{color}
\usepackage{float}
\usepackage{hyperref}
\usepackage{enumerate}
\usepackage{amsthm}
\usepackage{epsfig}
\usepackage[english]{babel}
\usepackage[latin1]{inputenc}

\newcommand{\p}{\mathrm{p}}

\renewcommand{\gg}{\mathfrak{g}}
\newcommand{\gl}{\mathfrak{gl}}

\newcommand{\hh}{\mathfrak{h}}
\newcommand{\kk}{\mathfrak{k}}
\renewcommand{\ll}{\mathfrak{l}}
\newcommand{\D}{\mathrm{D}}

\newcommand{\uu}{\mathfrak{u}}

\renewcommand{\ss}{\mathfrak{s}}
\newcommand{\pp}{\mathfrak{p}}
\renewcommand{\sl}{\mathfrak{sl}}
\newcommand{\so}{\mathfrak{so}}
\renewcommand{\sp}{\mathfrak{sp}}

\newcommand{\cH}{\mathcal{H}}
\newcommand{\cI}{\mathcal{I}}

\newcommand{\cM}{\mathcal{M}}

\newcommand{\cO}{\mathcal{O}}
\newcommand{\cP}{\mathcal{P}}

\newcommand{\cS}{\mathcal{S}}
\newcommand{\cT}{\mathcal{T}}

\newcommand{\cU}{\mathcal{U}}
\newcommand{\cY}{\mathcal{Y}}

\newcommand{\R}{\mathbb{R}}

\newcommand{\C}{\mathbb{C}}
\newcommand{\E}{\mathrm{E}}
\newcommand{\T}{\mathbb{T}}
\newcommand{\pr}{\mathrm{pr}}

\newcommand{\w}{\omega}
\newcommand{\Sl}{\mathrm{Sl}}

\newcommand{\So}{\mathrm{So}}

\newcommand{\K}{\mathrm{K}}

\renewcommand{\L}{\mathrm{L}}

\renewcommand{\T}{\mathrm{T}}
\newcommand{\U}{\mathrm{U}}
\newcommand{\G}{\mathrm{G}}

\renewcommand{\H}{\mathrm{H}}

\newtheorem{theorem}{Theorem}[section]
\newtheorem{lemma}[theorem]{Lemma}
\newtheorem{proposition}[theorem]{Proposition}
\newtheorem{corollary}[theorem]{Corollary}
\newtheorem{definition}[theorem]{Definition}
\newtheorem{example}[theorem]{Example}
\newtheorem{remark}[theorem]{Remark}

\title[The topology of tridiagonal isospectral sets]{The topology of tridiagonal isospectral sets}
\author{David Mart\'inez Torres}
\address{Department of Applied Mathematics, ETSAM Section, Universidad Polit\'ecnica de Madrid,
Avda. Juan de Herrera 4, 28040 Madrid, Spain}
\email{df.mtorres@upm.es}

\begin{document}

\begin{abstract} We describe the topology of the set of tridiagonal symmetric real matrices with fixed non-simple spectrum. The tridiagonal isospectral set is not a submanifold of the corresponding orthogonal conjugacy class, but a collection of submanifolds  which intersect cleanly. We show that the tridiagonal isospectral set is an aspherical space   and that its Betti numbers are combinatorial numbers. This  generalizes to arbitrary spectrum constructions and results known for simple spectrum. We also extend these results to split real semisimple Lie algebras.
 \end{abstract}
\maketitle

\section{Introduction}

Tridiagonal symmetric real matrices are classical objects in mathematics. There are not just ubiquitous in
numerical linear algebra but  they also have have  connections with many different subjects, such as difference equations, discrete integrable systems, and orthogonal polynomials, among others. On the one hand, tridiagonal symmetric matrices are elementary enough so that robust algorithms can be applied to them. On the other hand, they are general enough for instance to appear as adjacency matrices of complex systems in mathematical physics.

The topology of tridiagonal  symmetric real matrices with fixed simple spectrum was described in \cite{To}. It was proven that the integral homology  of the tridiagonal isospectral manifold is torsion free, that its Betti numbers are combinatorial numbers, and that the tridiagonal isospectral manifold is  aspherical and it is covered by euclidean space. These results were generalized in \cite{Da2} to analogs of tridiagonal isospectral manifolds for arbitrary split real semisimple Lie algebras.

The purpose of this note is to describe the topology of the {\bf tridiagonal isospectral set}
for any non-simple spectrum,
\[\cT=\{Q^T\Lambda Q\in \H\,|\,Q\in \So\},\]
where $\H\subset \sl$ is the subspace of traceless  upper Hessenberg real matrices, $\Lambda\in \sl$ is a fixed traceless  diagonal matrix, and $\So$ is the special orthogonal group.

Tridiagonal symmetric (real) matrices with non-simple spectrum are necessarily reducible  ---they have at least a zero non-diagonal entry--- and thus they split into block irreducible matrices. From the differential topology view, this means that for $\Lambda$ a diagonal matrix with non-simple spectrum its  tridiagonal isospectral set $\cT$
is a union of products of tridiagonal isospectral manifolds.

\begin{example}\label{ex:RP2} Let $\Lambda$ be the diagonal $3\times 3$ matrix with ordered entries $\{-2,1,1\}$. Its orthogonal conjugacy class
 is diffeomorphic to  $\mathbb{RP}^2$. It contains the  diagonal matrices with ordered entries $\Lambda$,  $\{1,1,-2\}$ and $\{1,-2,1\}$. The tridiagonal isospectral set is the join of two circles.
\begin{center}
\begin{figure}[h]
\includegraphics[scale=0.8]{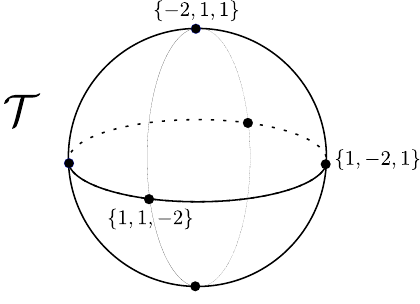}
 \caption{$\mathcal{T}$ is the image in $\mathbb{RP}^2$ of  the above two  great circles.}
\end{figure}
\end{center}
\end{example}
%

A  submanifold {\bf arrangement} of a manifold is a collection of (embedded) submanifolds with  clean intersections: The intersection of each subset of submanifolds must be a submanifold whose tangent space is the intersection of the respective tangent spaces.
Each submanifold of the arrangement is called a submanifold component.

Our first contribution fills in the details of a folklore result.

\begin{proposition}\label{pro:arrangement0} The tridiagonal isospectral set is a submanifold arrangement of its orthogonal conjugacy class. Each of its submanifold components is a product tridiagonal isospectral manifold.
\end{proposition}

 To discuss the homology of the tridiagonal isospectral set, we recall in some detail  the classical result for the simple spectrum case. Let $\Lambda_{n}$ be the diagonal matrix whose ordered entries are the ordered set $\overrightarrow{n}=\{1,2,\dots,n\}$ (that the trace be non-zero is irrelevant). Diagonal matrices in its orthogonal conjugacy class $\cO_{n}$  correspond to permutations. For a permutation $P\in \cS_n$ an ascent is given by consecutive indices such that $P(i)<P(i+1)$. In \cite{To} it was proven that the integral cohomology of the tridiagonal isospectral manifold $\cT_n$ is torsion free and its  $d$-th Betti number
is given by the number of permutations with $d$ ascents.

Consider  the ordered multiset
\begin{equation}\label{eq:multiset}\overrightarrow{m}=\{1,\dots,1,2,\dots,2,\dots,k,\dots k\},
 \end{equation}
 where $i$ appears $n_i$ times and $\sum_i n_i=n$. A permutation of ${\overrightarrow{m}}$ is an ordered list whose elements are the ones in $\overrightarrow{m}$. An ascent for it is defined as in the case of ordinary permutations. The ordered spectrum of an orthogonal conjugacy class $\cO\subset \sl$ defines an ordered multiset.

 Our second result states that the integral homology is encoded by ascents, regardless of the spectrum being simple or not.

\begin{theorem}\label{thm:betti} The integral homology of the tridiagonal isospectral set  is torsion free. Its $d$-th Betti number equals the number of permutations with $d$ ascents of the ordered multiset defined by its spectrum. If its spectrum is in bijection with $\overrightarrow{m}$ in \eqref{eq:multiset}, then by a result of MacMahon \cite{MM}
 \[\beta_d(\cT)=\sum_{j=0}^{d+1} (-1)^j\begin{pmatrix}n+1\\j\end{pmatrix}\begin{pmatrix}n_1+d-j\\n_1\end{pmatrix}\cdots \begin{pmatrix}n_k+d-j\\n_k\end{pmatrix}.\]
\end{theorem}

The main tool to prove Theorem \ref{thm:betti} is the same as in the simple spectrum case: The Toda vector field together with a  known associated Lyapunov function $f:\cO\to \R$.
This is a Morse function on $\cT$ as well, in the sense that it is Morse on each submanifold component. Its critical points are by definition the critical points of its restriction to the submanifold components. At each critical point ---a diagonal matrix--- cells on several submanifold components are attached. In principle, this might be  a  complicated attaching of a collection of intersecting cells. As it turns out, all cells are contained in a larger one, so the situation is analogous to that of Morse functions on manifolds, where a CW-structure is built by attaching just one cell for each critical point. This allows for encoding the homology of $\cT$  by means of combinatorics of permutations of multisets.


Our third result generalizes the homotopy type description to the non-simple spectrum case.

\begin{theorem}\label{thm:aspherical} The tridiagonal isospectral set  is aspherical.
\end{theorem}

For simple spectrum, the homotopy type of $\cT_n$ is analyzed by means of a  mirror structure $\mathcal{M}(\Delta_n)$ defined as follows \cite{To,Da2}. The group $\E$ of sign matrices modulo determinant acts on $\cT_n$ with fundamental domain $\Delta_n$ homeomorphic to a permutahedron. Each face of the permutahedron is a mirror and has an associated reflection. This produces a Coxeter system $W_n$ together with an epimorphism $W_n\to \E$. The so-called basic construction associates to the mirror structure a space
$\cU(W_n,\cM(\Delta_n))$ acted upon by $W_n$ in a proper fashion. This space, tiled by permutahedra, is the universal covering space of $\cT_n$ with group of Deck transformations
the kernel of the epimorphism. By a result of Davis \cite{Da1}, that $\widetilde{\cT}_n$ is contractible follows from the contractibility of the permutahedron, of its faces (mirrors), and of the intersection of its faces.

We shall give three different proofs of Theorem \ref{thm:aspherical}. The quickest one is of algebraic topological nature. It is a consequence of a result of Davis \cite{Da2} that implies that the intersection of submanifold components of $\cT$ is aspherical and its fundamental group injects in the fundamental group of each submanifold component.

In the second approach,
to the tridiagonal isospectral set we shall associate a mirror structure $\cM(\Delta)$, where $\Delta$ is a fundamental domain of the action of $\E$ on $\cT$, with its corresponding Coxeter system, epimorphism, and associated space
\[W\to \E,\quad \cU(W,\cM(\Delta))\cong \widetilde{\cT}.\]
The delicate point is that $\Delta$ is no longer homeomorphic to a permutahedron, or to a (convex) polytope. Rather, it is homeomorphic to a polyhedral complex made of permutahedra of different dimensions. We shall prove the contractibility of $\Delta$, of the mirrors, and of their intersections by means of the  Lyapunov Morse function $f$.

The third approach is by means on an embedding.
The permutations $\cS_n$ and $\cS^{\overrightarrow{m}}$ of the ordered set $\overrightarrow{n}$  and of the ordered multiset $\overrightarrow{m}$, respectively,  are partially ordered sets. The (right)  weak Bruhat order declares
that a permutation precedes another one if the inversion set of the former is included the inversion set of the latter. The standardization map
\[\mu:\cS^{\overrightarrow{m}}\to \cS_n\]
is an inclusion of posets whose image is a principal order ideal ---all permutations that precede a given permutation (see for instance \cite{BB}).
This points to a  different proof of the contractibility of $\Delta$ and of the intersections of the mirrors in $\cM(\Delta)$: To produce an identification of $\Delta$ with  the face subcomplex of the permutahedron defined by the principal order ideal $\mu(\cS^{\overrightarrow{m}})$, which is  known to be contractible \cite{BW}. Some combinatorics provide such a homeomorphism. We shall go further and provide a differential geometric construction
that also accounts for the Toda vector field.

\begin{theorem}\label{thm:mirror-embedding} The standardization map extends to a canonical embedding
\[\mu:\cT\to \cT_n\]
compatible with the Toda vector field.
The mirror structure $\cM(\Delta)$ on  $\cT$ is obtained by intersecting  the mirror structure $\cM(\Delta_n)$ with $\mu(\cT)$.
In particular, we obtain the commutative diagram
\[ \xymatrix{ \widetilde{\cT} \ar[r]\ar[d]     &\ar[d] \widetilde{\cT}_n\cong \R^{n-1} \\
   \cT\ar[r]^{\mu}&  \cT_n
}\]
where the top arrow embeds the contractible arrangement in euclidean space equivariantly (with respect to  $W$ and $W_n$), and vertical arrows are covering maps.
\end{theorem}

\begin{example}\label{ex:Cayley}
 Let $\Lambda=\{-2,1,1\}$ be as in Example \ref{ex:RP2}. The universal covering space of the
  join of two circles $\cT$ is the Cayley graph $G$ of the free group on two generators.

 The tridiagonal isospectral manifold $\cT_{3}$ is the surface of genus two $\Sigma$.  Its universal covering space is the hyperbolic  disk $\mathbb{H}^2$. The permutahedron $\mathcal{P}_{3}\cong \Delta_3$ is the hexagon. Its Coxeter system is
 \[W_3=\{w_i,\,1\leq i\leq 6,\,|\,w_i^2=1,\, (w_iw_{i+1})^2=1,\,(w_6w_1)^2=1,\, 1\leq i\leq 6\}.\]
 The basic construction tiles  $\mathbb{H}^2$ by hexagons with vertices of valence four.

The embedding $\mu:\cT\to \cT_{3}$ takes  $\Delta$ to the face subcomplex complex of the principal order ideal of $(231)$\footnote{The identification of $\Delta_3\cong \mathcal{P}_3$ the interchanges the (weak) right and left Bruhat orders. Our data in the example is for the latter, so strictly speaking we consider the principal filter of $(231)$.}: The union of the two edges of the hexagon connecting the vertices $(123)$, $(132)$ and $(231)$. The mirror structure $\cM(\Delta)$ comes from the four edges of the hexagon with non-empty intersection with $\mu(\Delta)$,
\[\cM(\Delta)=\{(123),\, [(123),(132)],\, [(132),(231)],\, (231)\}.\]
\begin{center}
\begin{figure}[H]
\includegraphics[scale=1]{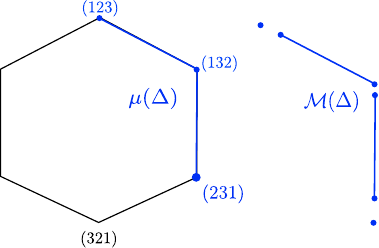}
\end{figure}
\end{center}
The Coxeter system $W\subset W_n$ is generated by the four consecutive reflections associated to the previous four edges. The ensuing commutative diagram embeds the Cayley graph in the hyperbolic disk
\[ \xymatrix{ G \ar[r]\ar[d]     &\ar[d] \mathbb{H}^2 \\
   \cT\ar[r]^{\mu}&  \Sigma
}.\]
Geometrically, we take the fundamental domain  $\Delta_{3}\subset \mathbb{H}^2$ and paste all the copies indexed by $W$, keeping track of how the two edges $\mu(\Delta)$ in the hexagon propagate.
\end{example}

We also present generalizations of some of the previous results in two directions.
On the one hand, we extend
to complex and quaternion matrices  the computation of the homology of the tridiagonal isospectral set.
On the other hand, we extend to split real semisimple Lie algebras all the results on the topology of the tridiagonal isospectral set.

The results of \cite{Da2} for tridiagonal isospectral manifolds in the general Lie algebra setting can be understood as ``differential geometric manifestations/enhancements'' of combinatorial properties of the weak Bruhat order on Weyl groups. From this perspective, our results can be understood as ``differential geometric manifestations/enhancements'' of combinatorial properties of the the weak Bruhat order on parabolic quotients of Weyl groups. Perhaps not surprisingly, the use of the Toda vector field predates the geometric counterparts of the combinatorial constructions.

\subsection*{ Acknowledgments} The author wishes to thank Francisco Santos and Carlos Tomei for valuable conversations and inputs. The author acknowledges  financial support by MCIN-AEI grant PID2022-139069NB-I00.


\section{Basic differential topology of the tridiagonal isospectral set}\label{sec:coordinates-tridiagonal}
In this section we  discuss elementary properties of the differential topology of the tridiagonal isospectral set. We do it by using some combinatorics, together with  the Toda vector field. We also introduce coordinates around diagonal matrices in the orthogonal conjugacy class that are adapted to the tridiagonal isospectral set: They map it to an arrangement of coordinate subspaces  and they are compatible with the Toda vector field. All matrices in this section have real entries.

The Toda vector field $\cY$ on $\sl$ \cite{F} is defined by
\begin{equation}\label{eq:Toda}
X'=[X,\pi_\so X],\quad \mathrm{I}=\pi_\so+ \pi_{\uu},\quad X\in \mathfrak{sl},
\end{equation}
where the identity map  on $\mathfrak{sl}$ is decomposed as the sum of projections
onto traceless skew-symmetric matrices $\so$, and traceless upper triangular matrices $\uu$. A fundamental result of Symes \cite{Sy} is the almost explicit computation of its trajectories,
\[X(t)=\kappa(\mathrm{e}^{tX_0})^TX_0 \kappa(\mathrm{e}^{tX_0}),\quad t\in \R,\qquad \kappa:\Sl\to \So,\]
where $\Sl$ is the special linear group and $\kappa$ is the first projection of the $\mathrm{QR}$ factorization. This result is the reason  why we will be favoring anticonjugation over conjugation.

The Toda vector field is tangent to many relevant submanifolds of $\sl$ and this makes it useful to study their topology. In particular,  if $\Lambda$ is a diagonal matrix with simple spectrum, then $\cY$ is tangent to the tridiagonal isospectral manifold that contains $\Lambda$. The reason is that $\cY$ is not just tangent to conjugacy classes, but also to the smaller orthogonal and upper triangular conjugacy classes.
Because symmetric matrices are preserved by conjugation by orthogonal matrices and Hessenberg matrices are preserved by conjugation by upper triangular matrices, $\cY$ is tangent to their intersection, the subspace  of tridiagonal symmetric matrices $\ss_\H$, and also to any tridiagonal isospectral manifold.

As recalled in the Introduction,  tridiagonal symmetric matrices with non-simple spectrum must be reducible, that is, they must have some zero subdiagonal entry. This implies that the tridiagonal isospectral set
\[\cT=\{Q^T\Lambda Q\in \ss_\H\,|\,Q\in \So\},\]
is a union of products of tridiagonal isospectral manifolds.

To describe the structure of $\cT$ in more detail we need some combinatorics for bookkeeping purposes.

By a partition $\p$ of $\sl$ we mean a collection of consecutive positions $(i_1,i_{1}+1),\dots ,(i_s,i_s+1)$. It defines a corresponding block subalgebra $\sl(\p)$ in which its block tridiagonal symmetric matrices sit, all of which are reducible. By a partition of $\sl$  at a diagonal matrix  $\Lambda$ we mean that $\Lambda$ is regarded as the corresponding block matrix in $\sl(\p)$ (the partition is ``evaluated'' at $\Lambda$). In a more combinatorial manner, if we identify $\Lambda$ with the list of its eigenvalues $\{\lambda_1,\dots,\lambda_n\}$ according to their positions,  then $\p(\Lambda)$ is a partition of the list into consecutive blocks so that $\lambda_{i_{j}},\lambda_{i_j+1}$ are in adjacent blocks of $\p$.
 A partition of $\sl$ is called simple at  $\Lambda$ if the eigenvalues on each block are different.

 To a reducible tridiagonal symmetric matrix $X$ we associate the partition $\p_X$ of $\sl$ given by the positions of its zero subdiagonal entries. If  $X\in \cT$, then we define $\cT_{\p_X}$ to be the intersection of $\ss_\H$ with the $\So(\p_X)$-conjugacy class of $X$,
 \[\cT_{\p_X}=\{Q^TXQ\in \ss_\H\,|\, Q\in \So(\p_X)\}.\]
 Because each block of $X$ is irreducible, the spectrum of the block must be simple. This implies that  $\cT_{\p_X}$ is a product tridiagonal isospectral manifold contained in $\cT$; if $\Lambda$ is a diagonal matrix in $\cT_{\p_X}$, then $\p_X$ is necessarily simple at $\Lambda$.

 A  {\bf simple partition}  of $\cT$
 is defined as an equivalence class of partitions of $\sl$ that are simple at some  diagonal matrix in $\cT$, where $\p(\Lambda)$ and $\p(\Lambda')$ are equivalent if their eigenvalues on each block are the same, regardless of the order. We still use the notation $\p$ to refer to a simple partition of $\cT$; it is understood that there is a choice of $\Lambda$ that we denote by $\p=\p(\Lambda)$, if need be. (And there might be other choices of $\Lambda$ leading to different simple partitions of $\cT$ for the same partition of $\sl$). For such a simple partition $\p=\p(\Lambda)$ the $\So(\p)$-conjugacy class of $\Lambda$ intersected with $\ss_\H$ is  a product tridiagonal isospectral manifold $\cT_\p$. Because the Toda vector field is tangent to the subspace of block matrices $\sl(\p)\subset \sl$, it is also tangent to $\cT_\p$.

 The collection of partitions of $\sl$ simple at $\Lambda$ (resp. simple partitions of $\cT$) is denoted by $\mathfrak{P}(\Lambda)$ (resp. $\mathfrak{P}$). The subset of maximal  partitions of this collection is denoted by $\mathfrak{M}(\Lambda)$ (resp. $\mathfrak{M}$), where a partition precedes another one it the former refines the latter.

The following result is an expanded version of Proposition \ref{pro:arrangement0} in the Introduction.
\begin{proposition}\label{lem:arrangement} The tridiagonal isospectral set $\cT$ is a submanifold arrangement of its orthogonal conjugacy class. Each submanifold component is a product tridiagonal isospectral manifold that corresponds to a maximal simple partition of $\cT$,
 \begin{equation}\label{eq:array}\cT=\bigcup_{\p\in \mathfrak{M}}\cT_{\p}.
  \end{equation}
\end{proposition}
\begin{proof}
 Let $X\in \cT$. Because $X\in \cT_{\p_X}$,
 \eqref{eq:array} holds.

Partitions of $\sl$ can be intersected. So simple partitions at $\Lambda$ can. However, if $\p_1,\dots,\p_s$ are simple partitions of $\cT$, then
\[\p=\p_1\cap \cdots \cap \p_s\in \mathfrak{P}\Longleftrightarrow \cT_{\p_1}\cap\cdots \cap \cT_{\p_s}\neq \emptyset.\]
To substantiate the equivalence, note that the statement in the left hand side means that $\Lambda\in \cT$ diagonal exists such that $\p_i=\p_i(\Lambda)$, so the implication from left to right is clear. If $X\in \cap_i\cT_{\p_i}$, then the closure of its Toda trajectory contains a diagonal matrix $\Lambda$.

Regarding the clean intersection property,  if $\Lambda$ and $\Lambda'$ belong to $\cap_i\cT_{\p_i}$, then $\p$ is simple at both. Therefore,  there is a unique permutation that preserves each block of $\p$ whose transpose takes   $\Lambda$ to $\Lambda'$, and this implies
\[\bigcap_i\cT_{\p_i}=\cT_{\p}.\]
The right hand side is a tridiagonal isospectral manifold, and, therefore, a submanifold of the orthogonal conjugacy class $\cO$.

The statement for tangent spaces is a consequence of the equality $\So(\p)=\cap_i \So(\p_i)$.  Vectors in $T_X\cT_{\p}$ are represented by curves in $\So(\p)$ starting at the identity that upon (anti)conjugation at $X$ produce a curve in $\ss_\H$. Such a curve  belongs  to all $\So(\p_i)$, so
$T_X\cT_\p= \cap_i T_X\cT_{\p_i}$.
\end{proof}

\begin{example}\label{ex:arrangement}
We describe the maximal partitions of the tridiagonal isospectral  set $\cT$ and its submanifold components in some low dimensional cases.
\begin{enumerate}[(i)]
 \item $\{-2,1,1\}$: \hskip .2cm $\cT=\mathbb{S}^1\cup \mathbb{S}^1$
 \[\p_1=\{-2,1\},\{1\},\quad \p_2=\{1\},\{-2,1\}\]
 \begin{center}
\begin{figure}[H]
\includegraphics[scale=1.4]{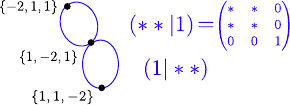}
 \caption{The isospectral set for $\Lambda=\{-2,1,1\}$.}
\end{figure}
\end{center}
 \item $\{-1,-1,1,1\}$:\hskip .2cm $\cT=\mathbb{S}^1\cup \mathbb{T}^2\cup \mathbb{S}^1$
 \[\p_1=\{-1\},\{-1,1\},\{1\},\quad \p_2=\{-1,1\},\{-1,1\},\quad \p_3=\{1\},\{-1,1\},\{-1\}\]
 \begin{center}
\begin{figure}[H]
\includegraphics[scale=1.4]{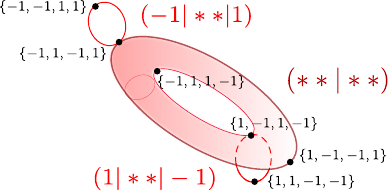}
\caption{The isospectral set for $\Lambda=\{-1,-1,1,1\}$.}
\end{figure}
\end{center}
  \item $\{-1,-1,0,2\}$:
  \[ \xymatrix @C=.01cm @R=.25cm{ &  \mathbb{T}^2 & && & &  & &  \p_1=\{-1,0\},\{-1,2\}  &      \\
\Sigma & \bigcup  & \Sigma & & & & & \p_2=\{-1\},\{-1,0,2\} &  & \p_3=\{-1,0,2\},\{-1\}\\
 & \mathbb{T}^2 & & & & &   &  & \p_4=\{-1,2\},\{-1,0\} &
}\]

  The two tori have empty intersection, each intersects a genus two surface $\Sigma$ in a circle, the two genus two surfaces also intersect in a circle, and the diagonal matrices corresponding to $\{-1,0,2,-1\},\{-1,2,0,-1\}$ are triple intersection points.
\end{enumerate}
\end{example}
\subsection{Adapted coordinates}\label{ssec:adapted-coordinates}
We construct local coordinates for the orthogonal conjugacy class $\cO$ around $\Lambda$ adapted to $\cT$,  meaning that $\cT$ corresponds to an arrangement of vector subspaces. The key properties of these adapted local coordinates are, on the one hand, the compatibility with the Toda vector field and, on the other hand, that this is an arrangement  of coordinate subspaces. Each coordinate subspace  corresponds to a submanifold component through $\Lambda$, which in turn corresponds to a maximal simple partition at $\Lambda$,
\[\mathfrak{M}(\Lambda)\ni \p\longleftrightarrow \cT_\p.\]

Let $\L$ denote unit lower triangular matrices.
A known consequence of Symes' factorization is
that  the map
\[\psi: \L\to  \cO,\quad L\mapsto \kappa(L)^T\Lambda \kappa(L),\quad \kappa:\Sl\to \So,\]
 takes the the trajectories of the action of $\exp(t\Lambda)$ on $\L$ by conjugation to the trajectories of the Toda vector field. For non-simple spectrum this is not a chart for $\cO$:
The kernel of the differential of $\psi$ at the identity $\mathrm{I}$ is the subspace  strictly lower triangular matrices which commute with $\Lambda$.

To obtain a chart for $\cO$ we restrict $\psi$ to a suitable subspace. Let  $\lambda_i$ be the entry of $\Lambda$ in the position $(ii)$. We define the following affine subspaces of the space $\L$ of unit lower triangular matrices,
\begin{equation}\label{eq:chart-domain}
 \L^\Lambda=\{L\in \L\,|\, L_{ij}=0,\,\,\mathrm{if}\,\,\lambda_i=\lambda_j,\,i\neq j\},\quad \L^\Lambda_\H=\L^\Lambda\cap \H.
\end{equation}
Because the tangent space of $\L^\Lambda$ at $\mathrm{I}$ is complementary to the kernel of the differential of $\psi$, the restriction $\psi: \L^\Lambda\to  \cO$ is a local diffeomorphism around $\mathrm{I}$.

Next, we want to characterize $\psi^{-1}(\cT)\subset \L^\Lambda$. To that end we denote by   $\nu:\So\to \U$ be the projection onto the second factor of the $\mathrm{QR}$ factorization, and by $\ll$ the strictly lower triangular matrices.
We have a commutative triangle
\begin{equation}\label{cd:triangle} \xymatrix{  X=\kappa(L)^T\Lambda \kappa(L)\in  \cO \ar[rd]\ar[rr]^{\psi^{-1}} &     &\ar@{|->}[ld]\L^\Lambda\ni L \\
 &  \nu(L)^{-1}X\nu(L)\in \Lambda+\ll\ni L^{-1}\Lambda L  &
}
\end{equation}
Because $\H$ is stable by conjugation by $\U$ it follows that
\[X\in \ss_\H\Longleftrightarrow \nu(L)^{-1}X\nu(L)=L^{-1}\Lambda L\in \Lambda+(\ll\cap \H).\]
Therefore the restriction to $\psi^{-1}$ to $\cT$  around $\Lambda$ has image
\[\{L\in \L^\Lambda\,|\, L^{-1}\Lambda L\in \Lambda+(\ll\cap \H)\}.\]
For any $L\in \L$ the entries of $Z=L^{-1}\Lambda L$ are
\begin{equation}\label{eq:conj-diagonal}
 Z_{ij}=(\lambda_i-\lambda_j)L_{ij}+\sum_{i<l_1\cdots <l_s< j}(-1)^{s+1} (\lambda_i-\lambda_{l_1})L_{i,l_1}L_{l_1,l_2}\cdots L_{l_s,j}.
\end{equation}
The subset  $\psi^{-1}(\cT)\subset \L^\Lambda$ is not a collection of subspaces, but as we shall see, it is a graph over a collection of coordinate subspaces.

 By a {\bf repeated pair} for $\Lambda$ we mean
 a pair of indices $(i,j)$, $i<j$, such that  $\lambda_i=\lambda_j$.
Repeated pairs for $\Lambda$ can be nested, and  we are interested in  innermost repeated pairs
\[(i_1,j_1),\dots ,(i_a,j_a),\quad i_1<\cdots <i_a,\quad j_{l+1}>j_l.\]
The monomial ideal associated to $\Lambda$ is generated by
\begin{equation}\label{eq:monomial}\prod_{l=0}^{j_1-i_1-1}L_{i_1+1+l,i_1+l}=0\,,\dots\,,\prod_{l=0}^{j_a-i_a-1}L_{i_a+1+l,i_a+l}=0.
 \end{equation}
Its zero set, denoted by $\T^\Lambda$, is an arrangement of coordinate subspaces.

\begin{proposition}\label{pro:adapted-coordinates}
There exists unique constants $c_{ij}\in \R$, so that the diffeomorphism
\[\phi:\L^\Lambda\to \L^\Lambda,\qquad L_{ij}\mapsto \begin{cases}L_{ij},&\quad {i-j}=1\\
                L_{ij}-c_{ij}\prod_{i\leq l< j}L_{l+1,l}&\quad i-j\neq 1\end{cases},\]
 composed by the right with $\psi^{-1}$,
\[\Psi^{-1}=\phi\circ \psi^{-1}:\cO\to \L^\Lambda,\] has the following properties:
\begin{enumerate}
 \item it defines local coordinates for $\cO$ around $\Lambda$ adapted to $\cT$;
 \item it sends $\cT$ to $\T^\Lambda$, which is an arrangement of coordinate subspaces of $\L^\Lambda_\H$ whose components are in correspondence with $\mathfrak{M}(\Lambda)$,
\[\T^\Lambda=\bigcup_{\p\in \mathfrak{M}(\Lambda)}\T_{\p},\quad \T_\p=\T^\Lambda\cap \sl(\p);\]
 \item the restriction of $\Psi^{-1}$ to $\cT_\p$ takes the Toda vector field to  $L'=[L,-\Lambda]$.
\end{enumerate}
%
\end{proposition}
\begin{proof}
The constant $c_{ij}$ is determined by setting $Z_{ij}=0$ in \eqref{eq:conj-diagonal} and expressing $L_{ij}$ as a monomial in subdiagonal variables, using induction on $i-j$.
For $i-j=1$ we just have $L_{i+1,i}=0$ if $(i,i+1)$ is a repeated pair, expressing that subdiagonal variables not in $\L^\Lambda$ (the variables not in $\L^\Lambda_\H$) must be zero. We assume that we have solved the equation  for $i-j=k$, meaning that the variables corresponding to repeated pairs are zero, and the other ones are monomials in non-zero subdiagonal variables.  For $i-j=k+1$ we have two possibilities,
\begin{itemize}
\item if $(i,j)$ is a repeated pair, then we set $c_{ij}=0$.
 \item if $(i,j)$ is not a repeated pair, then we can solve the equation and $c_{ij}$ is a rational function on differences of distinct eigenvalues.
\end{itemize}
This choice of constants defines a  diffeomorphism $\phi$ of $\L$ that sends $\L^\Lambda$ to itself.

If $\pr_\H:\L^\Lambda\to \L^\Lambda_\H$ denotes the orthogonal projection, then
\begin{equation}\label{eq:Toda-proj}\phi\circ \psi^{-1}(\cT)=\pr_\H\circ \psi^{-1}(\cT).
 \end{equation}
The reason is that the image $\psi^{-1}(\cT)$ is the solution set of $Z_{ij}=0$ in  \eqref{eq:conj-diagonal}. If $(i,j)$ is not a repeated pair, then
\[L_{ij}=c_{ij}\prod_{i\leq l< j}L_{l+1,l},\]
so by construction $\phi$ projects $\psi^{-1}(\cT)$ orthogonally into $\L^\Lambda_\H$.
The image of the projection is described by the innermost repeated pairs. If $(i,j)$ is one such pair, then
\[ c_{ij}=\frac{\prod_{l=0}^{i-j-2}(\lambda_{j+l}-\lambda_{j+l+1})}{\prod_{l=0}^{i-j-2}(\lambda_{j+l}-\lambda_i)}\neq 0,\]
and the monomial equation \[\prod_{l=0}^{j-i-1}L_{i+1+l,i+l}=0,\]
is satisfied by $\Psi^{-1}(\cT)$. If the repeated pair is not innermost, the inductive expression of $L_{ij}$ on subdiagonal variables is of the form
\[L_{ij}=d_{ij}\prod_{l=0}^{j-i-1}L_{i+1+l,i+l}.\]
Because $(i,j)$ is not innermost a factor of the monomial is already zero, so it imposes no additional constraint.  Therefore around $\Lambda$
\[\Psi^{-1}(\cT)=\T^\Lambda.\]
A maximal simple partition at $\Lambda$ corresponds to a minimal choice of subdiagonal variables that appear in all monomials in the ideal of $\Lambda$ \eqref{eq:monomial}. This proves (1) and (2).

As pointed out, the map $\psi$ takes the diagonal linear vector field
\[L'=[L,-\Lambda],\,L\in \L,\]
to the Toda vector field on $\cO$. Because each each variable $L_{ij}$ in $\L$ evolves independently, the diagonal linear vector field is tangent to $\L^\Lambda$ and the projection $\pr_\H$ intertwines its trajectories. This, together with the identity \eqref{eq:Toda-proj}, proves (3).
\end{proof}
\begin{remark}\label{rem:essential-property}
 The map $\psi:\L^\Lambda\to \cO$ is not just a local diffeomorphism but a diffeomorphims everywhere defined in $\L^\Lambda$ (whose image is a dense open subset of $\cO$) \cite{MT3}, and so $\Psi:\T^\Lambda\to \cT$ is.  We do not need this property as our considerations will be local around $\Lambda$.
\end{remark}

\section{The homology of the tridiagonal isospectral set}\label{sec:homology}
In this section we use a classical Lyapunov Morse function on $\cT$ for the opposite of the  Toda vector field $-\cY$ to produce a standard homotopy near critical points.
In  Section \ref{ssec:homology} we use this function to build a CW-complex structure on $\cT$ that allows for a combinatorial description of its integral homology.


The Toda vector field on $\cO$ is tangent to $\cT$. We regard its restriction to $\cT$ as a vector field on $\cT$, in the sense that it is a (smooth) vector field on every submanifold component.
The linear function
\begin{equation}\label{eq:Lyapunov} f(X)=\sum_{i}a_iX_{ii},\quad a_{i+1}>a_{i},\quad a_i>0,
 \end{equation}
 is a  known Lyapunov Morse function for  $-\cY$ on symmetric matrices, and on any orthogonal conjugacy class $\cO$ \cite[Lemma 1]{Fa}. Its critical points on $\cO$ are  diagonal matrices.

Because the diffeomorphism in Proposition \ref{pro:adapted-coordinates} used to produce $\Psi$ out of $\psi$ is tangent to the identity, the Hessian of the pullback of $f$ by $\Psi$ and by $\psi$ agree.  This Hessian is the quadratic form
 \begin{equation}\label{eq:Hessian}\sum (a_i-a_j)(\lambda_j-\lambda_i)L_{ij}^2,\quad L\in \L^\Lambda.
  \end{equation}

We are going to modify slightly $f$ on $\cT$  near each critical point $\Lambda$ to agree  with its Hessian in  the adapted local coordinates. We pull  $f$ back by  $\Psi:\L^\Lambda\to \cO$ and then by the inclusion $\L^\Lambda_\H\to \L^\Lambda$. The result is the auxiliary function  $F_\Lambda:\L^\Lambda_\H\to \R$ on the coordinate subspace $\L^\Lambda_\H$ which strictly contains
\[\T^\Lambda=\bigcup_{\p\in \mathfrak{M}(\Lambda)}\T_{\p}.\]
By using a radial bump function supported in an annulus with radii $r,2r$, for small enough $r$,  we may assume that $F_\Lambda$ equals its quadratic expansion at $\Lambda$,
\begin{equation}\label{eq:Hessian2}
F_\Lambda(\Lambda)+\sum_{\lambda_i\neq \lambda_{i+1}} (a_{i+1}-a_i)(\lambda_{i}-\lambda_{i+1})L_{i+1,i}^2.
\end{equation}
Because each $\T_\p$ is a coordinate subspace, the radial coordinate restricts to the radial coordinate on  $\T_{\p}$.
This means that ${F_\Lambda}|_{\T_\p}$ near the origin equals the quadratic expansion of the pullback of $f$ by $\Psi:\T_\p\to \cT_\p$. We do this at every critical point. We abuse notation and denote still by $f:\cT\to \R$ the modification of the original $f$ than now on $\T_\p$ is a quadratic form near the origin. The result is a  Morse function on $\cT$, in the sense that it is a Morse function on each submanifold component. It is also a Lyapunov function for $-\cY$ on each submanifold component. In local coordinates the difference of the original $F_\Lambda$ with its quadratic expansion is of order 3. The bump function has differential bounded by $2/r$, since it is supported in an annulus of width $r$.


We introduce the auxiliary diagonal linear  vector field on $\L^\Lambda_\H$,
\[-Y=\sum_i (\lambda_{i}-\lambda_{i+1})L_{i+1,i}\frac{\partial}{\partial L_{i+1,i}}.\]
The subspace $\L^\Lambda_\H$ splits into the stable and unstable subspaces for $-Y$,
\[\L^\Lambda_\H=W^s(\Lambda)\oplus W^u(\Lambda).\]
If $e\subset W^s(\Lambda)$ is a suitable small ball, then we can perform the standard Milnor's homotopy  that takes $F_\Lambda^{-1}((-\infty,F_\Lambda(\Lambda)+2\epsilon])$ to  $F_\Lambda^{-1}((-\infty,F_\Lambda(\Lambda)-\epsilon])\cup e$, for appropriate $\epsilon>0$.

\begin{lemma}\label{lem:Milnor} For any $\epsilon>0$ small enough  the standard homotopy induces  a homotopy on $\cT$ supported near $f^{-1}(f(\Lambda))$,
\begin{equation}\label{eq:standard-homotopy}
  h_\Lambda:f^{-1}((-\infty, f(\Lambda)+2\epsilon])\to f^{-1}((-\infty, f(\Lambda)-\epsilon])\cup \bigcup_{\p\in \mathfrak{M}(\Lambda)}e_{\p},
\end{equation}
where $e_{\p}=e\cap \T_{\p}$.
\end{lemma}
\begin{proof}
To ease the notation, we  assume that $F_\Lambda$ vanishes at $0$ so it is a quadratic form.
We rename the subdiagonal coordinates on $\L^\Lambda_\H$. We denote by $S_k$ (resp. $U_l$) the rescaling of the stable (resp. unstable) coordinates for $-Y$, so that balls in the new coordinates correspond to ellipsoids with axis length given by the norm of the coefficients  of $F_\Lambda$  in \eqref{eq:Hessian2}.
For $\epsilon>0$ we consider  the level set $F_{\Lambda}^{-1}(-\epsilon)$ and the cell
\[e=\{(S,U)\,|\,U=0,\,\sum_k {S_k^2}\leq \epsilon\}.\]
The union  $F_\Lambda^{-1}((-\infty,-\epsilon])\cup e$ is homotopic to $F_\Lambda^{-1}((-\infty,2\epsilon])$. Milnor's standard homotopy is constructed as follows \cite{Mi}:
\begin{itemize}
 \item we thicken the cell to
\[\mathcal{E}=\left\{(S,U)\,|\sum_l {U_l^2}\leq \epsilon,\,\sum_k S_k^2\leq \epsilon+\sum_l U_l^2\right\},\]
that gives rise  to the manifold with corners
\[\{(U,S)\in F_\Lambda^{-1}(-\epsilon)\,|\, \sum_k{S_k^2}\geq 2\epsilon\}\cup \{(U,S)\in \mathcal{E}\,| \sum_l U_l^2=\epsilon\};\]
\item the auxiliary vector field $Y$ is transverse to the previous manifold with corners
and any trajectory of $Y$ starting at $F_\Lambda^{-1}(2\epsilon)$ intersects it (transversely);
\item we let
\[h_1:F_\Lambda^{-1}((-\infty,2\epsilon])\times [0,1]\to F_\Lambda^{-1}((-\infty,-\epsilon])\cup \mathcal{E}\]
be the homotopy that is the identity on $F_\Lambda^{-1}((-\infty,-\epsilon]\cup \mathcal{E}$, and that at any point $L$ not there  traverses at a constant speed the  integral curve of $Y$  from $L$ to the manifold with corners;
\item the thickening is collapsed back to the core by means of  the homotopy
\[h_2:F_\Lambda^{-1}((-\infty,-\epsilon])\cup \mathcal{E}\to F_\Lambda^{-1}((-\infty,-\epsilon])\cup e,\]
 that is the identity on $F_\Lambda^{-1}((-\infty,-\epsilon]\cup e$, and that at any point $L$ not there  traverses at a constant speed the unique closed segment from $L$  to $F_\Lambda^{-1}(-\epsilon)\cup e$ tangent to $\sum_l -U_l\frac{\partial}{\partial U_l}$;
\item the standard homotopy for $F_\Lambda$  is the concatenation  $h=h_1*h_2$.
\end{itemize}
We claim that any coordinate subspace $\T_\p\subset \L^\Lambda_\H$ is invariant by $h_1$ and $h_2$: First,  $-Y$ restricts to $\T_\p$ to the pullback by $\Psi$ of the Toda vector field. Second,  the auxiliary vector field $\sum_l -U_l\frac{\partial}{\partial U_l}$ is diagonal linear and thus it is tangent to any coordinate subspace.

If  $\epsilon <cr$, for some small $c$, then the points that $h_1$ sends to $\mathcal{E}$ lie inside the ball of radius $r$. By considering the subset
\[\Psi(\mathcal{E}\cap \T^\Lambda)\subset \cT,\] we can define $h_1$ on $\cT$ by the same recipe that uses now $-\cY$. We can also  define $h_2$ on $\cT$ by means of $\Psi$. The result is a  homotopy on $\cT$ supported near $f^{-1}(f(\Lambda))$,
\[
  h_\Lambda:f^{-1}((-\infty, f(\Lambda)+2\epsilon])\to f^{-1}((-\infty, f(\Lambda)-\epsilon])\cup \bigcup_{\p\in \mathfrak{M}(\Lambda)}e_{\p},\quad e_{\p}=e\cap \T_{\p}.
\]
(We use the same notation for  $e_{\p}\subset \T_{\p}$ and for its image by $\Psi$).
\end{proof}
\subsection{The $\mathrm{CW}$-complex structure of the tridiagonal isospectral set}\label{ssec:homology}
The Morse function $f:\cT\to\R$ allows to build $\cT$ by attaching at each critical point a cell for each submanifold component through it. Some combinatorics shows that just one cell is attached, this providing a CW-complex structure on $\cT$ from which its homology can be computed.

A coordinate axis $L_{i+1,i}$ is in $W^s(\Lambda)\subset \L^\Lambda_\H$ if $\lambda_{i}<\lambda_{i+1}$, i.e., if $(i,i+1)$ is an ascent for $\Lambda$. Upon intersecting $\L^\Lambda_\H=W^s(\Lambda)\oplus W^u(\Lambda)$ with $\T_{\p}$, we obtain
\[\T_{\p}=W^s(\p)\oplus W^u(\p),\quad \p\in \mathfrak{M}(\Lambda),\]
the splitting into unstable and stable subspaces for the restriction of $-Y$ to $\T_{\p}$.

Let $\Lambda^+,\Lambda^-\in \cT$ be the diagonal matrices with eigenvalues  ordered increasingly and decreasingly, respectively. These are  a (the) source and a (the) sink for the Toda vector field on $\cT$.

\begin{lemma}\label{lem:runs} The unstable and stable subspaces at $\Lambda$ have the following properties:
\begin{enumerate}
\item  There exist $\p_M,\p_m\in \mathfrak{M}(\Lambda)$ such that
 \begin{equation}\label{eq:max-unstable}W^s(\p_M)=W^s(\Lambda), \quad W^u(\p_m)=W^u(\Lambda).
  \end{equation}
 In particular for  any $\p\in \mathfrak{M}(\Lambda)$
 \begin{equation}\label{eq:stable-inclusion}
 W^s(\p)\subset W^s(\p_M),\quad W^u(\p)\subset W^u(\p_m).
 \end{equation}
\item  All $W^u(\p)$ (resp. $W^s(\p)$) are trivial if and only if $\Lambda$ is the maximum $\Lambda^+$
(resp. minimum $\Lambda^-$) for $f$ in $\cT$.
 \end{enumerate}

 \end{lemma}
\begin{proof}
 We regard partitions $\p\in \mathfrak{M}(\Lambda)$ as minimal sets of subdiagonal variables solving the system of monomial equations \eqref{eq:monomial}.
 First, we consider contiguous repeated pairs $(i,i+1)$ and select the corresponding variable $L_{i+1,i}$. These variables are not in $\L^\Lambda_\H$, so they do not contribute to $W^s(\Lambda)$. Let $(i,j)$ be an innermost repeated pair with $i-j>1$. There has to be an index $i\leq l< j$ such that $\lambda_l>\lambda_{l+1}$. Otherwise we would have
 \[\lambda_i<\lambda_{i+1}<\cdots <\lambda_{j-1}<\lambda_j,\]
 and that is a contradiction. We select the variable $L_{l,l+1}$. By doing  this for every innermost repeated pair we obtain a solution to the system of monomial equations, that is, a partition simple at $\Lambda$. We select $\p_M\in \mathfrak{M}(\Lambda)$ which the former partition precedes. Because all ascents for $\Lambda$ are within the blocks of $\p_M$,
 \[W^s(\p_M)=W^s(\Lambda),\]
 and this settles (1) for stable subspaces.

Next, we prove that if $\Lambda\neq \Lambda^+$,
then $W^u(\Lambda)\neq \{0\}$. Let $\lambda_l>\lambda_{l+1}$. We claim that there is a partition simple at $\Lambda$ that does not contain the variable $L_{l+1,l}$. If this is the case, then a maximal partition simple at $\Lambda$ preceded by it has non-trivial unstable manifold. To prove the claim  we let $(i,j)$ be an innermost repeated pair. (By assumption $i-j>1$). Because its associated monomial has  degree greater than one, it contains at least one variable different from $L_{l+1,l}$, that we select.

The case for unstable manifolds and for the minimum is proved analogously.
\end{proof}

\begin{corollary}\label{cor:morse-ordering} The function $f:\cT\to \R$
\begin{enumerate}
 \item does not have local maxima and minima different from the global ones;
 \item if $\Lambda$ is the maximum (resp. minimum) for $f$ on $\cT_{\p}$, then it cannot be the maximum (resp.
 minimum) for $f$ on any other submanifold component through $\Lambda$. Moreover,
 \begin{itemize}
  \item $\p_M\in \mathfrak{M}(\Lambda)$ (resp. $\p_m$) is unique and equals $\p$;
  \item $W^s(\p_M)$ (resp. $W^u(\p_m)$) strictly contains the  stable (resp. unstable) manifolds for all the other $\cT_{\p'}$, $\p'\in \mathfrak{M}(\Lambda)$.
 \end{itemize}
\end{enumerate}
\end{corollary}
\begin{proof}
 The absence of non-trivial local maxima and minima follows from item (2) in Lemma \ref{lem:runs}.

 Let $\Lambda$ be the maximum for $f$ on $\cT_{\p}$. If $W^s(\p)\subset W^s(\p_M)$, then $\cT_{\p}\subset \cT_{\p_M}$ because the stable manifold for the maximum is dense in $\cT_{\p}$ and the submanifold components are compact. Because in the arrangement a submanifold cannot be strictly contained in another one, $W^s(\p_M)$ has to contain strictly the other stable manifolds for all partitions in $\mathfrak{M}(\Lambda)$. This also implies that it cannot be the maximum of another for the restriction of $f$ to another submanifold component.
\end{proof}

We need another tool for the analysis of the attaching maps for the cells built out of $f$: The group $\E$ of sign matrices modulo determinant. It acts by conjugation on $\cO$ and on $\L^\Lambda$
preserving $\cT$ and $\T^\Lambda$, respectively. It also takes the Toda vector field to itself. The group $\E$ is generated by the fundamental reflections which act on $\cT$ and $\T^\Lambda$ by changing the sign of just one off-diagonal position. The local diffeomorphism $\psi^{-1}:\cO\to \L^\Lambda$ is equivariant with respect to the action of $\E$, and so the projection $\pr_\H:\L^\Lambda\to \L^\Lambda_\H$ is. Therefore $\Psi:\T^\Lambda\to \cT$ is equivariant with respect to the action of $\E$.

\begin{proof}[Proof of Theorem \ref{thm:betti}] By item (1) in Lemma \ref{lem:runs},  the Toda vector field and the function $f$ produce a CW-complex decomposition of $\cT$: Just one cell is attached at each critical point $\Lambda$. The homotopy $h_\Lambda$ defined in Lemma  \ref{lem:Milnor} retracts $f^{-1}((-\infty,f(\Lambda)+2\epsilon])$ into
\[f^{-1}((-\infty, f(\Lambda)-\epsilon])\cup e_{\p_M},\]
where the dimension of the cell $e_{\p_M}$ is  the number $s$ of ascents of $\Lambda$.

Next, we claim that the attaching map
\[\partial e_{\p_M}\to e_{\p'_M}/\partial e_{\p'_M},\quad \p'\in \mathfrak{M}(\Lambda'),\]
where  $e_{\p'_M}$ is a cell of dimension $s-1$, has degree zero. Assuming the claim,
we conclude  that the integral cohomology is torsion free and the $d$-th Betti number is the number of permutations with $d$ ascents.

Let us assume that the spectrum of $\cT$ has $k$ distinct elements with multiplicities $n_1,\dots,n_k$.
By an identity of MacMahon \cite{MM} (see also \cite[Page 45]{Knu})
 \[\beta_d(\cT)=\sum_{j=0}^{d+1} (-1)^j\begin{pmatrix}n+1\\j\end{pmatrix}\begin{pmatrix}n_1+d-j\\n_1\end{pmatrix}\cdots \begin{pmatrix}n_k+d-j\\n_k\end{pmatrix}.\]

To substantiate the claim we recall that the attaching map is obtained by flowing $\partial e_{\p_M}$ with $\cY$ until it hits $f^{-1}(f(\Lambda')+2\epsilon))$, then composing with the homotopy $h_{\Lambda'}$, and then collapsing to a point the complement of the interior of $e_{\p'_M}$.

The stable space $W^s(\p_M)$ is larger than $W^s(\p'_M)$, so there exists a fundamental reflection $\sigma\in \E$ that acts non-trivially on $W^s(\p_M)$ and that fixes $W^s(\p'_M)$. Because the action of $\E$ preserves $\cY$, the Toda vector field is tangent to the fixed-point set $\cT^\sigma$ (to each submanifold component there). The homotopy $h_{\Lambda'}$ from Lemma \ref{lem:Milnor} also preserves the coordinate subspace ${(\T^{\Lambda'})}^\sigma$. Therefore only points in the hypersurface $\partial {e_{\p_M}^\sigma}\subset \partial e_{\p_M}$ can be sent to the interior of $e_{\p'_M}$. In the preimage of a small neighborhood of $\Lambda'$ the attaching map is the composition of a flow with a smooth retraction, and thus it is smooth. By invariance of domain, the attaching map cannot be surjective, and this implies  that its degree is zero.
\end{proof}
\begin{remark}\label{rem:closure} The closure of the unstable submanifold at $\Lambda$ is a submanifold of $\cT_{\p_M}$:  The product tridiagonal isospectral manifold  defined by the partition whose contiguous blocks occur at the ascents of $\Lambda$ \cite[Section 9]{Da2},\cite[Proposition 2]{Fr}. Therefore, the generators of the homology of $\cT$ are represented by submanifolds.
\end{remark}

\section{The homotopy of the tridiagonal isospectral set}\label{ssec:homotopy}
In this Section we prove that $\cT$ is an aspherical space by different methods. First, by using algebraic topological methods. Second, by constructing a mirror structure with appropriate properties \cite{Da1} which are verified using the Morse function $f$. Third,  by means of a canonical embedding into the tridiagonal isospectral  manifold $\cT_{n}$ (with simple spectrum), compatible with the Toda vector field.

\subsection{The universal covering space and graphs of euclidean spaces}

We recall that the fundamental reflections of the  group $\mathrm{E}$ are those changing sign in just one off-diagonal entry. Partitions of $\sl$ correspond to subgroups $\E_\p$ generated by those fundamental reflections indexed by the block changes of the partition.

The tridiagonal isospectral set for $\Lambda=\{-2,1,1\}$ is the join of two circles. It defines an oriented graph whose three vertices are the two circles and their intersection, and whose two arrows correspond to the inclusions of the intersection on each circle.
The universal covering space of the tridiagonal isospectral set is the Cayley graph for the free group on two generators (see Examples \ref{ex:RP2} and \ref{ex:Cayley}). This is a tree which  consists of copies of the real line with certain pattern in their intersection, which is either empty or a point (contractible). The intersection pattern can be partially described by an oriented graph where vertices corresponds to lines and to intersections of lines ---euclidean spaces of dimension one or zero--- and oriented edges correspond to inclusions. (To encode completely the Cayley graph we would need to order the incoming edges to a vertex corresponding to a line).

The following is an improved version of Theorem \ref{thm:aspherical} which includes a generalization of the properties of intersections for the Cayley graph.

\begin{theorem}\label{thm:aspherical-algebraic-topology}
 The universal covering space of $\cT$ is contractible. Under the covering map $\widetilde{\cT}\to \cT$
 \begin{enumerate}
  \item the preimage of each submanifold component $\cT_\p\subset \cT$ is a collection indexed by $\pi_1(\cT)/\pi_1(\cT_\p)$ of spaces diffeomorphic to euclidean space;
  \item the intersection of connected components of the preimages of $\cT_{\p_1},\dots,\cT_{\p_s}$, if non-empty, is diffeomorphic to euclidean space.
 \end{enumerate}
\end{theorem}
\begin{proof}
To the tridiagonal isospectral set $\cT$ we assign the oriented graph whose vertices are the partitions $\p\in \mathfrak{P}$ which are intersections of maximal partitions of $\cT$, and whose arrows correspond to the inclusion of the respective $\cT_\p$.
Let $\p,\p'\in \mathfrak{M}$ be vertices of the graph corresponding to maximal partitions such that
\[\cT_{\p}\cap \cT_{\p'} \neq\emptyset.\]
Then $\cT_{\p\cap \p'}\subset \cT_\p$ is a connected component of the fixed-point
of the action of $\E_{\p'}$ on $\cT_\p$. By \cite[Lemma 7.1]{Da2}, applied to product tridiagonal submanifolds, the map between fundamental groups induced by the inclusion  is a monomorphism. The same holds for $\p_1,\dots,\p_s\in \mathfrak{M}$. The intersection, if non-empty, is aspherical
and
\[1\to \pi_1(\cT_\p)\to \pi_{1}(\cT_{\p_i}),\quad \p=\p_1\cap\cdots\cap \p_s.\]
If to each vertex $\p$ we associate $\pi_1(\cT_\p)$ and to each arrow (inclusion) the corresponding monomorphism of fundamental groups, then we obtain a graph of groups. By \cite[Theorem 1B.11]{Ha} $\cT$ is an aspherical space and the inclusion of each submanifold component induces a monomorphism in fundamental groups
\[1\to \pi_1(\cT_\p)\to \pi_{1}(\cT).\]
Because $\cT$ is a CW-complex with trivial homotopy groups, it is contractible. The previous monomorphism on fundamental groups
 implies that the preimage under $\widetilde{\cT}\to \cT$ of $\cT_\p$ is a disjoint collection of spaces parametrized by $\pi_1(\cT)/\pi_1(\cT_\p)$, each of which is diffeomorphic to $\widetilde{\cT}_\p$. By \cite[Theorem 3.3, Remark 3.4, and Example 4.4]{Da2} (see also  \cite[Section 5]{To}) this universal covering space is diffeomorphic to euclidean space. This proves (1).

For any subset  $N\subset \mathfrak{M}$ so that the union of its submanifold components is a connected $\cT_N\subset \cT$, it follows from the proof of \cite[Theorem 1B.11]{Ha} that $\cT_N$ is aspherical and that the inclusion induces a monomorphism of fundamental groups,
\[1\to \pi_1(\cT_N)\to \pi_{1}(\cT).\]
This implies that to analyze the intersection $\cap_i\widetilde{\cT}_{\p_i}$ of copies of the universal covering space of $\cT_{\p_i}$ in $\cT$, we can do it in (a copy in $\widetilde{\cT}$ of ) $\widetilde{\cT}_N$, $N=\{\p_1,\dots,\p_s\}$. We argue by induction in the number of submanifold components using the Cech-de Rham double complex (using reduced singular cohomology instead of forms). If $N$ has two members $\p_a,\p_b$, we consider the infinite cover of $\widetilde{\cT}_N$ by copies $U_\alpha,V_\beta$ of $\widetilde{\cT}_{\p_a}$ and  $\widetilde{\cT}_{\p_b}$, respectively  (strictly speaking, to have an open cover we would take small tubular neighborhoods of $\cT_{\p_a}$ and $\cT_{\p_b}$ in $\cO$, intersect them with $\cT_N$, and take their preimage by the universal covering map). Because the triple intersections are empty, the spectral sequence that converges to the reduced homology
$\widetilde{H}^*(\widetilde{\cT}_N)$ degenerates into a Mayer-Vietoris type sequence. The contractibility of the subsets in the cover implies
\[\widetilde{H}^*(\widetilde{\cT}_N)\cong \sum_{\alpha,\beta}\widetilde{H}^*(U_\alpha\cap V_\beta),\]
and this implies that all intersection are either empty or connected. The connected intersections are copies of the universal covering space of $\cT_{\p_a}\cap \cT_{\p_b}$. One can argue  by induction on the number of members in $N$.  Using that the intersections of less than $s$ subsets of the cover is contractible, that  the intersection of $s+1$ subsets is empty, and that $\widetilde{\cT}_N$ is contractible, the spectral sequence
implies that the intersections of $s$ subsets of the cover is either empty or connected (a copy of $\cap_i\widetilde{\cT}_{\p_i}$).
\end{proof}

\begin{remark} We can assign to $\cT$ the corresponding oriented graph of intersections where vertices correspond to the different components of the  $\widetilde{\cT}_\p$ and their intersections ---euclidean spaces by Theorem \ref{thm:aspherical-algebraic-topology}.
We do not know whether a combinatorial description of $\widetilde{\cT}$ is possible beyond low dimensional cases. A first step  would be to analyze whether the universal covering space $\cT_{\p_a}\cap \cT_{\p_a}$  sits in $\widetilde{\cT}_{\p_a}$, up to isotopy, as a vector subspace in euclidean space (whether its fundamental group at infinity is trivial).
\end{remark}

\subsection{The mirror structure}\label{ssec:mirror}
The {\bf positive chamber} for the action of $\E$ on $\cT$ is defined as those matrices with positive subdiagonal entries,
\[\Delta=\{X\in \cT\,|\, X_{i+1,i}\geq 0\}.\]
It is a fundamental domain for the action of $\E$ on $\cT$, in the sense that every orbit of $\mathrm{E}$ intersects $\Delta$ at exactly one point.

We use $f$ to describe the (trivial) topology of the positive chamber.

\begin{proposition}\label{pro:contractible} The positive chamber $\Delta$ is contractible.
 \end{proposition}
\begin{proof}
We contract $\Delta$ to the minimum $\Lambda^-$ by adapting the
standard homotopy $h_\Lambda$ to the positive chamber.

%
We denote by $\Delta^\Lambda\subset \T^\Lambda$ the union of positive chambers for the action of $\E$ on $\T_\p$, $\p\in \mathfrak{M}(\Lambda)$. The equivariance of $\Psi:\T^\Lambda\to \cT$ implies that $\Delta^\Lambda$ is mapped (in)to $\Delta$. The positive chambers $\Delta$ and $\Delta^\Lambda$ are invariant by the Toda vector field and by $Y$, respectively; $\Delta^\Lambda$ is also invariant by $\sum_l -U_l\frac{\partial}{\partial U_l}$. Therefore we can restrict the standard homotopy

\begin{equation}\label{eq:fundamental-standard-homotopy}
  h_\Lambda:f^{-1}((-\infty, f(\Lambda)+2\epsilon])\cap \Delta^\Lambda\to \left(f^{-1}((-\infty, f(\Lambda)-\epsilon])\cup e_{\p_M}\right)\cap \Delta^\Lambda.
\end{equation}
The intersection $W^s(\p_M)\cap \Delta^\Lambda$ is a positive chamber for the residual action of $\mathrm{E}$ on $W^s(\p_M)$. It has positive dimension  unless $W^s(\p_M)$ is a point. By Corollary \ref{cor:morse-ordering} this only happens at $\Lambda^-$. If this is not the case, then we consider the auxiliary constant diagonal vector field
\[\sum_{k}\epsilon_k\frac{\partial}{\partial S_k},\]
where $\epsilon_k=0$ if $S_k$ is not a coordinate in $W^s(\p_M)$, and it is 1 or $-1$ otherwise, so that it points inside the positive chamber $W^s(\p_M)\cap \Delta^\Lambda$ along the part of its boundary which is a union of coordinate subspaces. We define the homotopy
\[h_3:\left(f^{-1}((-\infty, f(\Lambda)-\epsilon])\cup e_{\p_M}\right)\cap \Delta^\Lambda\to f^{-1}((-\infty, f(\Lambda)-\epsilon])\cap \Delta^\Lambda\]
that is the identity  at a point in $f^{-1}((-\infty, f(\Lambda)-\epsilon])\cap \Delta^\Lambda$, and that at  $L\in e_{\p_M}\cap \Delta^\Lambda$ traverses at a constant speed the segment tangent to the constant auxiliary vector field that goes from $L$ to the component of the boundary of  $e_{\p_M}\cap \Delta^\Lambda$ in $\partial e_{\p_M}$.

We denote the concatenation of $h_\Lambda$ with $h_3$ by $h_{\Delta^\Lambda}$. Observe that $h_{\Delta^\Lambda}$ is well-defined in $\Delta$, since $h_3$ takes $\cT_{\p_M}\cap \Delta$ to itself.

We use  the Toda vector field combined with the homotopy $h_{\Delta^\Lambda}$
near each critical point to contract $\Delta$ into $\Lambda^{-}$.
We start at the maximum $\Lambda^+$.
There, we apply $h_3$ to retract $\Delta$ onto $f^{-1}((f(\Lambda^+)-\epsilon))\cap \Delta$. Then we follow the flow of the Toda vector field until we hit the level set
$f^{-1}(f(\Lambda)+2\epsilon))\cap \Delta$, where $\Lambda$ is the next critical point. If $\Lambda$ is not the minimum, then  we can apply  $h_{\Delta^\Lambda}$ to retract onto $f^{-1}((f(\Lambda)-\epsilon))\cap \Delta$. By induction, we retract $\cT$ onto $\Lambda^{-}$.
\end{proof}

For each fundamental reflection $\sigma\in \E$  we consider the intersection of its  fixed-point set with $\Delta$,
\[\Delta^\sigma\subset \cT^\sigma.\]
We want to analyze the topology of its connected components and of their intersections, for different fundamental reflections.

We  denote by $\cS^\cT$ the collection of diagonal matrices in $\cT$.
\begin{proposition}\label{pro:mirror}  Each connected component of $\Delta^\sigma$ is a fundamental chamber of a product  tridiagonal isospectral set, and so are the intersections of these connected components for different fundamental reflections, if non-empty.
 \end{proposition}
\begin{proof}
The reflection $\sigma$ changes sign on a given subdiagonal position $X_{i+1,i}$. Its fixed-point set on $\cT$ is
\[\cT^\sigma=\{X\in \cT\,|\, X_{i+1,i}=0\}.\]
The reflection defines the corresponding partition $\p_\sigma$ of $\sl$. It gives rise to a collection $\p_{\sigma,j}$ of equivalence classes, where $\p_\sigma(\Lambda)\sim \p_\sigma(\Lambda')$, for $\Lambda,\Lambda'\in\cS^\cT$,  if on both blocks the (unordered) eigenvalues of $\Lambda$ and $\Lambda'$ agree.
These equivalence classes  are in bijection with the connected components of $\Delta^\sigma$: The intersection of the  $\So(\p_\sigma)$-conjugacy class of a representative $\Lambda$ with $\ss_\H$ is the product of two tridiagonal isospectral sets $\cT_{\p_{\sigma,j}}$, and thus, by Theorem \ref{thm:betti},  a connected subset. The intersection $\Delta^\sigma_j=\Delta^\sigma\cap \cT_{\p_{\sigma,j}}$ is the positive chamber for the action on $\cT_{\p_{\sigma,j}}$ of the subgroup of $\E$ generated by the fundamental reflections different form $\sigma$. Therefore
by Proposition \ref{pro:contractible}, applied to the product tridiagonal set $\cT_{\p_{\sigma,j}}$,
it is connected.
Any $X\in \Delta^\sigma$ must be in one of the $\Delta^\sigma_j\subset \cT_{\p_{\sigma,j}}$, according to the spectrum of each of its two blocks.
If $j\neq l$, then $\cT_{\p_{\sigma,j}}\cap \cT_{\p_{\sigma,l}}=\emptyset$. The reason is that
both $\cT_{\p_{\sigma,j}}$ and  $\cT_{\p_{\sigma,l}}$ sit in the block product $\sl(\p_\sigma)\subset \sl$ and the eigenvalues of a matrix in $\cT_{\p_{\sigma,j}}$ in one block (in both) differ from the eigenvalues of a matrix in $\cT_{\p_{\sigma,l}}$. The conclusion is that we have a disjoint union
\[\Delta^\sigma=\coprod_j \Delta^\sigma_j.\]

Consider a subset of fundamental reflections $\sigma_1,\dots,\sigma_s$ of $\mathrm{E}$. Select for each one a tridiagonal isospectral set and let  $\cT_{\p_i}$, $\Delta^i$, and $\p$ denote $\cT_{\p_{\sigma_i,j}}$, $\Delta\cap \cT_{\p_{\sigma_i,j}}$, and  $\cap_i\p_{\sigma_i}$, respectively.
Let $X\in \cap_i \cT_{\p_i}\neq \emptyset$. Then
 the closure of the trajectory of $\cY$ through it contains a diagonal matrix $\Lambda$. This implies that
$\cap_{i} \cT_{\p_{\sigma_i}}$ contains $\cT_\p$ defined as the intersection of the $\So(\p)$-conjugacy class of $\Lambda$ with $\ss_\H$, which is a product tridiagonal isospectral set.
To show the equality, let $\Lambda'$ be another diagonal matrix in the intersection. This means that
the repeated eigenvalues within each block of $\p$ are permuted within the block. Therefore there exists a permutation which preserves each block  of $\p$ and whose transpose takes $\Lambda$ to $\Lambda'$.
By construction $\cap_{i} \Delta^i$ is the positive chamber for the action on $\cT_\p$ of the subgroup of $\E$
generated by fundamental reflections different from $\sigma_i$.
\end{proof}
\begin{remark}\label{rem:connected-intersection}
The proof of Proposition \ref{pro:mirror} shows that for arbitrary partitions $\p_i$, if we select  a connected component $\cT_{\p_i,j}$ of the fixed point-set for the action on $\cT$ of the subgroup of $\E$ generated by the fundamental reflections indexed by the block changes in $\p_i$, then
\[\bigcap_i\cT_{\p_i,j}=\cT_\p,\quad \p=\bigcap_i\p_i,\]
if non-empty. This extends the statement in the proof of Lemma \ref{lem:arrangement} from simple partitions of $\cT$ to  partitions $\p_i$ of $\sl$.
\end{remark}

We have the tools to define the mirror structure and analyze its properties.

\begin{definition}\label{def:mirror} The {\bf mirror structure} $\cM(\Delta)$ on $\Delta$ associated to $\cT$ has mirrors $\{\Delta^{\sigma_i}_j\}$, the connected components of the fixed point sets $\Delta^{\sigma_i}$ for the fundamental reflections.
\end{definition}

\begin{example}\label{ex:mirror1} We describe the mirror structures for two of the tridiagonal  isospectral sets in Example \ref{ex:arrangement}. The fundamental domain is in black and the color of a mirror corresponds to the fundamental reflection of $\E$ that defines it.
\begin{enumerate}[(i)]
 \item $\Lambda=\{-2,1,1\}$: $\cT=\mathbb{S}^1\cup \mathbb{S}^1$
\begin{figure}[H]
\includegraphics[scale=1.4]{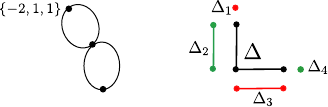}
\caption{The mirror structure for $\Lambda=\{-2,1,1\}$ has four mirrors.}
 \end{figure}
 \item $\Lambda=\{-1,-1,1,1\}$: $\cT=\mathbb{S}^1\cup \mathbb{T}^2\cup \mathbb{S}^1$
 \begin{figure}[H]
\includegraphics[scale=1.4]{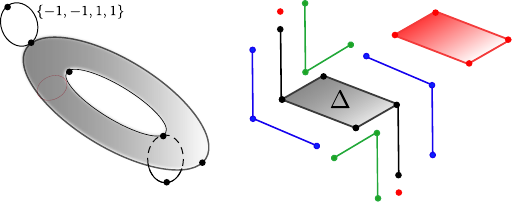}
\caption{The mirror structure for $\Lambda=\{-1,-1,1,1\}$ has seven mirrors.}
 \end{figure}
 \end{enumerate}
\end{example}

Following  \cite[Section 6]{Da1},\cite[Section 1]{Da2} to $\cM(\Delta)$ one associates a (right-angled) Coxeter system  by fixing a bijection $b$ from a set $S$ to the mirrors \[b:S=\{w_a\}\to \{\Delta^{\sigma_i}_j\},\] and defining
\[W=\{S\,|\, w_a^2=1,\, (w_{a_l}w_{a_s})^2=1\Longleftrightarrow b(w_{a_l})\cap b(w_{a_s})\neq \emptyset \}.\]
There is a well-defined surjective morphism
 \begin{equation}\label{eq:torsion-free} \chi:W\to \mathrm{E},\quad \w_a\mapsto \sigma_{w_{a}},\quad  (b(w_a)=\Delta_{a}^{\sigma_{w_a}}).
   \end{equation}
The {\bf basic construction} \cite[Section 13]{Da1},\cite[Section 1]{Da2} produces a topological space acted upon by $W$,
\[\cU(W,\cM(\Delta))=W\times \Delta/\sim, \quad (w,X)\sim (w',X')\Longleftrightarrow \sigma_w^{-1}\sigma_{w'}X\sigma_w{\sigma_{w'}}^{-1}=X.\]

The next result includes a second proof  of Theorem \ref{thm:aspherical} in the Introduction.

\begin{theorem}\label{thm:mirror} The mirror structure $\cM(\Delta)$ has the following properties:
 \begin{enumerate}
  \item (Finite type) The group $W$ acts properly on $\cU(W,\cM(\Delta))$;
 \item $\cU(W,\cM(\Delta))$ is contractible;
  \item The kernel of $\chi:W\to \mathrm{E}$
  is a torsion free group isomorphic to $\pi_1(\cT)$.
  \item  $\cT$ is aspherical and  $\cU(W,\cM(\Delta))$ is a contractible arrangement.
 \end{enumerate}
\end{theorem}
\begin{proof}
Let $X\in \Delta$ and let $W_X\subset W$ be the subgroup generated by the letters that corresponds to mirrors of $\cM(\Delta)$ containing $X$. Because any two such mirrors intersect at $X$, the morphism $\eqref{eq:torsion-free}$ maps $W_X$ isomorphically  to the centralizer of $X$ for the action of $\E$. Hence  $W_X$ is finite, and this proves (1).

Let $S'\subset S$ be such that $W_{S'}$ is finite. We claim that $\chi$  maps it isomorphically to the subgroup of $\E$ that fixes all points in the intersection of mirrors corresponding to $S'$. Because $W_{S'}$ is finite any such pair of mirrors must have non-empty intersection. Thus, the elements in $S'$ are mapped by $\chi$ to different fundamental reflections. In particular $\chi$ maps $W_{S'}$ isomorphically to the subgroup generated by them. Let us order the fundamental reflections $\sigma_{i_1},\dots,\sigma_{i_s}$, $i_1<\dots <i_s$, where the index is the last position of the first block. Let $\Lambda_{i_1}$ be
a diagonal matrix of size $i_1$ whose spectrum is the one determined by the first block of the partition $\p_{i_1}$ of $\cT$ defined by the mirror $b(w_{i_1})$. Because all mirrors intersect the first one, the spectrum of $\Lambda_{i_1}$ is contained in the first block of the spectrum determined by $\p_{i_j}$, $j>1$. We can then take a second diagonal matrix of size $i_2-i_1$ whose spectrum, together with the one of $\Lambda_{i_1}$, is the spectrum determined by the first block of $\p_{i_2}$. If we juxtapose the two block diagonal matrices, then  we get a block diagonal matrix whose spectrum is contained in the spectrum determined by the first block of $\p_{i_j}$, $j>2$. By induction, we obtain
\[\Lambda\in \bigcap_{j} b(w_{i_j}).\]
Every point in the intersection of mirrors is fixed by $\chi(W_{S'})$. Conversely, there cannot be another reflection $\sigma$ fixing all points in the intersection. The intersection is connected, so it would be contained in a  connected component of $\Delta^\sigma$, that is, in yet another mirror, what would make the intersection smaller.

By Propositions \ref{pro:mirror} and \ref{pro:contractible}, $\Delta$, the mirrors, and the non-empty  intersections of mirrors are contractible. Therefore,
by   \cite[Theorem 13.5]{Da1} $\cU(W,\cM(\Delta))$ is a contractible space, and this proves (2).

Any finite subgroup of $W$ is conjugated to a finite $W_{S'}$ \cite[Lemma 1.3]{Da2}. We saw that \eqref{eq:torsion-free} maps it isomorphically into a subgroup of $\mathrm{E}$, therefore the kernel $\Gamma$  is torsion-free.
Because $W$ acts properly, the torsion-free kernel $\Gamma$ acts freely on $\cU(W,\cM(\Delta))$, and therefore it is the universal covering space of the space
\[\cU(W,\cM(\Delta))/\Gamma,\]
with $\Gamma$ acting as the group of Deck transformations.
The continuous map
\[\cU(W,\cM(\Delta))\to \cT,\quad [w,X]\to \sigma_w X\sigma_w \]
is $W$-equivariant (the action on $\cT$ is induced by $\chi$), so it descends to an   $\mathrm{E}$-equivariant map
\[\cU(W,\cM(\Delta))/\Gamma\to \cT,\quad [w\Gamma,X]\to \sigma_{w\Gamma}X\sigma_{w\Gamma}.\]
Because $\Delta$  ---which is identified by the map--- is a fundamental chamber for the action of $\mathrm{E}$ on both spaces, the map is a homeomorphism. Therefore $\Gamma$ is the fundamental group of $\cT$, which is aspherical, and the covering map lifts the arrangement structure  from $\cT$ to
$\cU(W,\cM(\Delta))$ (the collection of manifolds with clean intersections, disregarding that they are submanifolds of an ambient manifold).
\end{proof}
\begin{remark}\label{rem:connected-intersection-2}
The proof of Theorem \ref{thm:mirror} refines Remark \ref{rem:connected-intersection}, in the sense that for
\[\bigcap_i\cT_{\p_i,j}=\cT_\p,\quad \p=\bigcap_i\p_i,\]
just non-empty pairwise intersections are required.
\end{remark}
\subsection{Embedding into  the tridiagonal isospectral manifold and the permutahedron}\label{ssec:embedding}
In this section we construct a canonical embedding of $\cT$ into $\cT_{n}$ which is compatible with the Toda vector field and with the mirror structures $\cM(\Delta)$ and $\cM(\Delta_{n})$; in fact, $\cT_{n}$ can be replaced by an appropriate tridiagonal isospectral set. The embedding, defined by means of parallel transport, restricts to diagonal matrices to the standardization map.

We start by recalling the standardization map.
On $\cS_n$ and on the set $\mathcal{S}^{\cT}$ ---identified with the  permutations of the ordered multiset defined by the spectrum  $\cT$--- we consider the weak (right) Bruhat order. For a permutation an inversion is a pair of indices $i<j$ such that $\lambda_i>\lambda_j$. The inversion set is the collection of pairs of indices corresponding to inversions. One declares $\Lambda\leq \Lambda'$ if the inversion set of $\Lambda$ is included in the inversion set of $\Lambda'$.  The standardization map   \cite[Theorem 8]{BB}  is an order preserving injection
\begin{equation}\label{eq:order-embdedding}
\mu: \cS^{\cT}\to \mathcal{S}_{n},\quad \Lambda\mapsto  \mu(\Lambda),
\end{equation}
where one replaces the $n_1$ appearances in order of the smallest eigenvalue by $1,\dots,n_1$, then the appearances in order of the second smallest eigenvalue by $n_1+1,\dots, n_1+n_2$, and so on.
Moreover, $\mu(\mathcal{S}^{\cT})$ is the  principal order ideal of the element obtained by writing the blocks from the last to the first one
\[n_1+\cdots +n_{k-1}+1,\dots ,n_1+\cdots +n_{k},n_1+\cdots +n_{k-2}+1,n_1,\dots n_1+\cdots +n_{k-1},\dots\]

We provide a geometric version of $\mu$ that  embeds $\cT$ into $\cT_{n}$. This canonical embedding  comes by parallel transport with respect to a natural Ehreshmann connection. The construction requires some elementary analysis of the action of $\E$ on the subspace of symmetric tridiagonal matrices, and its relation to tridiagonal isospectral sets.

The subspace $\ss_\H$ is stratified by the the orbit type stratification for the action of $\E$. A stratum is a subset of
 symmetric tridiagonal matrices with the same centralizer. Such centralizers are subgroups generated by the fundamental reflections $\sigma_{i_1},\dots,\sigma_{i_s}$. The closure of their corresponding stratum $\cI$ is the subspace
 \[\overline{\cI}=\{X\in \ss_\H\,|\, X_{i_1+1,i_1}=\cdots =X_{i_s+1,i_s}=0\}.\]
 The stratum $\cI$ is the open subset of matrices whose only zero off-diagonal entries are the ones in the equation above.
Strata, collections of fundamental reflections, and partitions are in correspondence.  Each stratum $\cI$ comes with the block subalgebra $\so(\p)$ of its corresponding partition.
Let $X\in \ss_\H$ and let $\cI$ denote the stratum containing $X$. The intersection
of the $\So(\p)$-conjugacy class of $X$ with $\ss_\H$ is a product tridiagonal isospectral submanifold $\cT_X$. (Each block matrix of $X$ has non-zero subdiagonal entries, and, therefore, the spectrum is simple within the block). The dimension of $\cT_X$ is $n-d$, where $d$ is the codimension of $\cI$ in $\ss_\H$ (or  $d+1$ is the number of blocks), and $\cT_X$ is contained in the closure of $\cI$. The strata in the closure of $\cI$ have to be added because the conjugacy class has matrices with more off-diagonal zeros; in particular, all diagonal matrices with the spectrum of $X$.

We regard the trace as an inner product on $\ss_\H$.
The restriction of the trace to a stratum is just the sum of the traces on its blocks.

\begin{definition}\label{def:Ehreshmann} Let $X\in \cI$. We define $H_X\subset \ss_\H$ to be the orthogonal complement in $\cI$ of the tangent space $T_X\cT_X$.
We denote the collection of this subspaces by $\mathcal{H}$. By construction, all subspaces have dimension $d$.
\end{definition}

We need yet another piece of structure defined by $\cI$. We consider the action of the permutation group  $\cS_n$ on traceless diagonal matrices $\D$. The negative (Weyl) chamber $\Delta_\D$ are  diagonal matrices with entries ordered increasingly. To a stratum/partition  we associate to the Young subgroup $\cS_{n,\p}$ of permutations preserving each block, and the negative chamber for the action of $\cS_{n,\p}$ on  $\D$,
\[\Delta_\p=\{ \lambda_1\leq \cdots\leq \lambda_{n_1},\lambda_{n_1+1}\leq \cdots \leq\lambda_{n_1+n_2},\dots, \lambda_{n_{k-1}+1}\leq\cdots \leq\lambda_n\}.\] We denote its interior by $\mathrm{int}(\Delta_{\p})$. Let
\[\lambda:\cI\to \mathrm{D},\quad X\mapsto \lambda(X),\]
 be the map that  sends $X$ to the diagonal matrix defined by the ordered eigenvalues of the first block, then the ordered eigenvalues of the second block, and so on.

The strata, the collection of tangent spaces $T_X\cT_X$, and  $\cH$ interact as follows.
\begin{lemma}\label{lem:Ehreshmann}   Let $\cI$ be a stratum of $\ss_\H$ for the action of $\E$. Then
\begin{enumerate}
 \item the map $\lambda:\cI\to \mathrm{D}$ is a submersion onto its image $\mathrm{int}(\Delta_{\p})$ with fiber $\lambda^{-1}(\lambda(X))=\cT_X\cap \cI$;
 \item the restriction of $\cH$ to $\cI$ is an Ehreshmann connection for $\lambda:\cI\to \mathrm{D}$;
 \item the parallel transport map is compatible with the Toda vector field.
\end{enumerate}
\end{lemma}
\begin{proof}
 For any  $X\in \cI$ the intersection of its $\So(\p)$-conjugacy class in $\sl$ with the closure $\overline{\cI}$ is the product tridiagonal isospectral manifold $\cT_X$. Because this is a tranverse intersection,
\[T_X\cT_X=\{[X,U]\,|\, U\in \so(\p),\,[X,U]\in \ss_\H\}.\]
 The submanifolds $\cT_X$, $X\in\cI$, define a foliation on an open subset of the closure of $\cI$. In particular, its  intersection with $\cI$ is a foliation  integrating  the tangent spaces in the definition of $\cH$. By construction, $\cH$ on $\cI$ is the (smooth) distribution of orthogonal subspaces (within $\cI$).

The map $\lambda:\cI\to \mathrm{D}$ is a product map where on each block a matrix is sent to the diagonal matrix with its ordered eigenvalues. This is a submersion onto the corresponding negative chamber of matrices with the given size with fiber the open subset of the tridiagonal isospectral manifold of matrices with non-zero off-diagonal entries in the block \cite[Remark 2 after Lemma 2.2]{To}. Therefore $\lambda:\cI\to \mathrm{D}$ is a submersion onto its image $\mathrm{int}(\Delta_\p)$.

The compatibility of the parallel transport map with the Toda vector field is equivalent to
$[\cY,\cH]\subset \cH$.
A matrix $V$ in the closure of $\overline{\cI}$ is in $\cH_X$ if and only if
\[\langle V,[X,U]\rangle=0,\quad  U\in \so(\p),\,[X,U]\in \ss_\H,\]
where the inner product is the trace. Because the tangent bundle of $\cT_X$ is involutive
\[\langle [[X,\pi_{\so}X],V],[X,U]\rangle=-\langle V,[[X,\pi_{\so}X],[X,U]]\rangle=-\langle V,[X,U_2]\rangle=0.\]
\end{proof}

We need a version of Lemma \ref{lem:Ehreshmann} that applies not just to $\cT_X\cap \cI$, but to the whole product isospectral manifold $\cT_X$.
For a stratum $\cI$ we denote by $\widehat{\cI}$ the result of adding the closure of the fibers of $\lambda$. In other words, for each $X\in \cI$ we take the whole product tridiagonal isospectral manifold $\cT_X$. This is an open subset of the closure $\cI$.
The proof of Lemma \ref{lem:Ehreshmann} shows that
\[\lambda:\widehat{\cI}\to \mathrm{int}(\Lambda_{\p})\]
is a submersion whose compact connected fibers are the $\cT_X$, $X\in \cI$.

\begin{proposition}\label{pro:Ehreshmann} The restriction of $\cH$ to $\widehat{\cI}$ is an Ehreshmann connection for
\[\lambda:\widehat{\cI}\to \mathrm{int}(\Lambda_{\p}).\]
Moreover, $\lambda$ and $\cH$ are equivariant with respect to the action of the subgroup  $\E_{\p}\subset \E$ that centralizes $\cI$ and of the Young subgroup $\cS_{n,\p}$ (the latter acting trivially on the base).
\end{proposition}
\begin{proof}
The collection of subspaces $\cH$ on $\ss_\H$ is $\E$-equivariant because  the action of $\E$ preserves each stratum of $\ss_\H$, each tridiagonal isospectral submanifold there, and the trace. 
The same applies for the action of the Young subgroup $\cS_{n,\p}$ on $\overline{\cI}$.

The fibers of $\lambda:\widehat{\cI}\to \Delta_\p$ foliate the open subset $\widehat{\cI}\subset \overline{\cI}$. We define its orthogonal complement $\cH_2$ with respect to the trace in $\overline{\cI}$. This is a smooth distribution which equals $\cH$ on $\cI$. We claim that $\cH_2$ agrees with $\cH$ on $\widehat{\cI}$, what would prove  the proposition.

Let $X\in \cI$ and $Y\in \cT_X$ a matrix in a stratum $\cI'$ different from $\cI$.  Let $\E_{\p'}\subset \E$ be the centralizer of $\cI'$. The subspace $\overline{\cI}$  splits orthogonally into the trivial isotypic block and the sum of the remaining isotypic blocks for the action of $\E_{\p'}$,
\[\overline{\cI}=\overline{\cI'}\oplus G.\]
The tangent space of $\cT_X$ at $Y$ is a submodule and therefore its analogous splitting is obtained by intersection
\[T_Y\cT_X=(T_Y\cT_X\cap \overline{\cI'})\oplus (T_Y\cT_X\cap G).\]
In particular $T_Y\cT_X\cap G$ is perpendicular to $\overline{\cI'}$. This means that
${\cH_2}$ at $Y$ is  the orthogonal complement with respect to the trace in $\cI'$ of the subspace
$T_Y\cT_X\cap \overline{\cI'}$. This tangent space is the one of $\cT_Y$. Near $Y$, taking the $\So(\p')$-conjugacy class of $Y$ and intersecting with $\cI'$ is the same as  taking its $\So(\p)$-conjugacy class and intersecting with $\cI'$. Therefore $\cH$ and $\cH_2$ agree at $Y$.
\end{proof}

\begin{definition}\label{def:parallel-transport} Let $X$ and $X'$ be points in $\cI$. We define the diffeomorphism
 \begin{equation}\label{eq:parallel-transport}
  \mu_\p:\cT_X\to \cT_{X'}
 \end{equation}
by parallel transport with respect to $\cH$ over the segment joining $\lambda(X)$ to $\lambda(X')$.
\end{definition}
\begin{remark}\label{rem:Ehreshman2}
We do not know whether the equidimensional family of subspaces $\cH$ is smooth everywhere in $\ss_\H$. We do not know either a closed formula for $\mu_\p$.
\end{remark}

 By Proposition \ref{pro:Ehreshmann} the diffeomorphism $\mu_\p$ is equivariant with respect to $\E_{\p}$. Also, the manifold $\widehat{\cI}$ contains the diagonal matrices $\D$, and $\D$ is an integral leaf of $\cH$.  For $P\in \cS_{n,\p}$ and $X,X'\in \cI$,  the segment
 $[\lambda(X),\lambda(X')]\subset \mathrm{int}(\Delta_{\p})$ lifts to the segment $[P^T\lambda(X)P,P^T\lambda(X')P]\subset \D$.

To promote $\mu_\p$ to a map defined on the tridiagonal isospectral set
\[\cT=\bigcup_{\p\in \mathfrak{M}}\cT_\p,\]
we have to choose appropriate intervals so that the parallel transport on submanifold components is compatible.
But before that, we need to discuss   the tridiagonal isospectral sets in which we can embed $\cT$.

The diagonal matrix $\Lambda^+\in \cT$ with spectrum ordered increasingly belongs to the negative chamber $\Delta_\D$. Its spectrum determines the stratum of the orbit type stratification of $\D$ with respect to the action of $\cS_n$ to which it belongs. We take $\Lambda_\star\in \D$ belonging to a stratum whose closure contains $\Lambda^+$ and denote by $\cT_\star$ the tridiagonal isospectral set it defines. Typically, $\Lambda_\star$ will be in the open stratum (different eigenvalues) so $\cT_\star$ is a tridiagonal isospectral manifold. We denote by $I\subset \Lambda_\D$ the segment joining $\Lambda^+$ to $\Lambda_\star$. Going along the segment in the opposite direction changes the eigenvalues of $\Lambda_\star$ in an affine fashion and for $t=0$ it collapses some adjacent eigenvalues to give $\Lambda^+$.

Let $\p\in \mathfrak{M}$ and let $\Lambda^+_\p$ denote the matrix in $\cT_\p\subset \cT$ with spectrum ordered increasingly. We select the permutation $P_\p\in \cS_n$ such that $P^T_\p$ takes the occurrences of the smallest eigenvalues in $\Lambda^+$ to its occurrences of in $\Lambda^+_\p$ in order, then the occurrences of the second smallest eigenvalue, and so on. (The permutation with the minimum number of transpositions taking $\Lambda^+$ to $\Lambda^+_\p$).
The submanifold component $\cT_\p$ sits inside some $\widehat{\cI}$. That $\Lambda_\star$ be in the closure of the stratum of $\Delta_{\D}$ containing $\Lambda^+$ is equivalent to $P^T\Lambda_\star P\in \mathrm{int}(\Delta_{\p})$.
Therefore we can consider  the parallel transport over $P^T_\p I P_\p$,
\[\mu_\p:\cT_\p\to \lambda^{-1}(P^T\Lambda_\star P)\subset \cT_\star.\]

The following result is an expanded version of Theorem \ref{thm:mirror-embedding} in the Introduction.
\begin{theorem}\label{thm:parallel-transport}
The map
\begin{equation}\label{eq:parallel-transport-isospectral}
 \mu:\cT\to \cT_\star,\quad \mu|_{\cT_\p}=\mu_\p
\end{equation}
is a well-defined $\E$-equivariant embedding that
\begin{enumerate}
\item takes the Toda vector field on $\cT$ to the Toda vector field on $\cT_\star$;
 \item  restricts to diagonal matrices to the standardization map $\cS^{\cT}\to \cS^{\cT_\star}$;
 \item takes the mirror structure $\cM(\Delta)$ to the mirror structure $\mu(\cT)\cap\cM(\Delta_\star)$,
 \begin{itemize}
  \item  a  mirror in $\mu(\Delta)$ is a non-empty intersection of $\mu(\Delta)$ with a mirror in $\Delta_\star$, this giving an injection $\mu:\cM(\Delta)\to \cM(\Delta_\star)$;
  \item an intersection of mirrors in $\cM(\Delta)$ is obtained by intersecting $\mu(\Delta)$ with the intersection of the mirrors in $\cM(\Delta_\star)$ that correspond to those in $\cM(\Delta)$ by $\mu$.
  \end{itemize}
\end{enumerate}
\end{theorem}
\begin{proof}
We must show that if $\p,\p'\in \mathfrak{M}$ and $\cT_\p\cap \cT_{\p'}\neq \emptyset$ (in particular $\p$ and $\p'$ are different partitions), then $\mu_\p$ and $\mu_{\p'}$ agree on the intersection $\cT_{\p\cap \p'}$.

Let $\Lambda^+_{\p\cap \p'}\in \cT_{\p\cap \p'}$ be diagonal matrix with eigenvalues ordered increasingly.
We go from $\Lambda^+_\p$ to $\Lambda^+_{\p\cap \p'}$ by applying (the transpose of a) permutation of $P_2$ in the Young subgroup $\cS_{n,\p}$. As we saw, the horizontal lift (with respect to $\mu_\p$) at $\Lambda^+_{\p\cap \p'}$ of the segment $P^T_\p IP_\p$  is the segment
\[P_2^TP_\p^T I P_\p P_2\subset \Delta_{{\p\cap \p'}}\]
that connects $\Lambda^+_{\p\cap \p'}$ to  $P_2^TP_\p^T \Lambda_\star P_\p P_2$. Because $P_2^TP_\p^T$ is the permutation with minimum number of transpositions taking $\Lambda^+$ to $\Lambda_{\p\cap \p'}^+$, it follows that
\[{\mu_\p}|_{\cT_{\p \cap\p'}}=\mu_{\p\cap \p'}.\]
Therefore
\[\mu:\cT\to \cT_\star\]
is a well-defined embedding.
The equivariance with respect to the action of $\E$ and the compatibility with the Toda vector field follow from Proposition \ref{pro:Ehreshmann}.

By construction, $\mu$ takes a matrix $\Lambda$ to the result of replacing the $n_1$ appearances in order of its smallest eigenvalue by the  the first $n_1$ eigenvalues of $\Lambda_\star$, the $n_2$ appearances in order its second smallest eigenvalues by the eigenvalues from $n_1+1$ to $n_1+n_2$ of $\Lambda_\star$, and so on, and this proves (2).

The equivariance of $\mu$ with respect to $\E$ implies that $\mu(\Delta)$ is mapped into the fundamental domain $\Delta_\star$ of $\cT_\star$. (Signs of subdiagonal entries cannot change during the transport along the intervals). For the same reason $\mu(\Delta^\sigma)\subset \Delta^\sigma$. A connected component $\Delta^\sigma_j$ corresponds to a partition $\p_{\sigma,j}$ of $\cT$; different mirrors  correspond to different choices of (possibly repeated) spectrum in the two blocks of $\p_\sigma$. Because $\mu$ is injective (on diagonal matrices), different choices for the spectrum of $\cT$ in the blocks of $\p_\sigma$ correspond to different choices for the spectrum of $\cT_\star$ in these blocks. This implies that a mirror of $\cM(\Delta)$
is sent to the unique connected component of the intersection of a mirror of $\cM(\Delta_\star)$ with $\mu(\Delta)$. Therefore we have an induced injection $\mu:\cM(\Delta)\to \cM(\Delta_\star)$. From this it follows that the intersection of a collection of mirrors  in $\cM(\Delta)$ is the intersection with $\mu(\Delta)$ with the mirrors in $\cM(\Delta_\star)$ that correspond to those in $\cM(\Delta)$ by $\mu$. Moreover, if a collection of mirrors of $\cM(\Delta)$ in the image of $\mu$ has non-empty intersection, then the diagonal matrix whose eigenvalues are ordered increasingly within the blocks of the partition corresponding to the intersection, belongs to $\mu(\Delta)$.
\end{proof}

Let $W_\star$, $\chi_\star$, $\cU(W_\star,\cM(\Delta_\star))$ be the Coxeter system, epimorphism and contractible space associated to $\cT_\star$. We assume that the bijection for the Coxeter system $W,\chi,\cU(W,\cM(\Delta))$ is defined by composing $\mu$ with fixed bijection for $W_\star$.

\begin{corollary}\label{cor:mirror-emedding}
The embedding $\mu$ induces a monomorphism of Coxeter systems
\begin{equation}\label{eq:Coxeter-embeddig-1}
 \theta_{\mu}:W\to W_\star,\quad w_k\mapsto b^{-1}(\mu(b(w_k)))
\end{equation}
and commutative diagrams
\[ \xymatrix{ W \ar[r]^{\theta_{\mu}}\ar[d]^{\chi}     &\ar[d]^{\chi_\star} W_\star & & & \widetilde{\cT}\cong\cU(W,\cM(\Delta)) \ar[r]^{\tilde{\mu}}\ar[d]  &\ar[d]\cU(W,\cM(\Delta_\star))\cong \widetilde{\cT}_\star\\
   E\ar[r]&  E & & &  \cT\ar[r]^{\mu}&  \cT_\star
},\]
the second of which has horizontal arrows that are embeddings, the top one exhibiting $\widetilde{\cT}$ as a submanifold arrangement of $\widetilde{\cT}_\star$.
\end{corollary}

\begin{proof}[Proof of Theorem \ref{thm:mirror-embedding}]
For simplicity's sake, in the introduction we worked with the tridiagonal isospectral manifold $\cT_{n}\subset \gl$.
The results in this section (and in the entire manuscript) hold if we drop the trace zero requirement and  replace $\sl$ by $\gl$ as ambient space.
\end{proof}

We regard the permutahedron $\mathcal{P}_n\subset \R^n$ as the convex hull of the vectors whose coordinates corresponds to permutations.

\begin{corollary}\label{cor:convex-embedding} The embedding $\mu$ induces a canonical homeomorphism from $\Delta$ to a subcomplex of the permutahedron $\mathcal{P}_n$
\[C:\Delta\mapsto \mathcal{P}_n,\quad X\mapsto \pr_\D(Q\Lambda_n Q^T),\quad \mu(X)=Q^T\Lambda_n Q,\]
\end{corollary}
\begin{proof}
The twisted Schur-Horn map
\[\Delta_n\to \cP_n,\quad X=Q^T\Lambda_n Q^T\mapsto \pr_\D(Q\Lambda_n Q^T)\]
is the canonical homeomorphism between $\Delta_n$ and $\cP_n$.

 A submanifold component $\cT_\p$ embeds in $\cT_n$ as a product tridiagonal submanifold inside $\sl(\p)$. The map $C$ sends $\cT_\p\cap \Delta$ to the corresponding product permutahedron. This is the convex hull of the image of the permutations in  $\mu(\cS^{\cT}\cap \cT_\p)$, which is a face of $\mathcal{P}_n$.
 Therefore $C(\Delta)$ is a subcomplex of $\mathcal{P}_n$.
\end{proof}

 \begin{remark}The homeomorphism $\Delta_\star\to \mathcal{P}_n$ sends a permutation to its inverse. If we work in $\Delta_\star$, then  we have to consider the right Bruhat order (what we have been doing so far). The faces of  $\Delta_\star$
are in bijection with standard parabolic cosets of $\cS_{n}$. To each coset in $\cS_n/\cS_{n,\p}$ there corresponds the interval $[P,PP_0]$, where $P_0$ is the longest element of $\cS_{n,\p}$  and $P$ is the shortest representative of the coset.
If we work with the convex polytope, then we must switch to the left weak Bruhat order, or relabel the vertices by value rather than by position (c.f. Example \ref{ex:Cayley}).
\end{remark}

By \cite[Theorem 8]{BB} $\mu(\cS^{\cT})\subset \cS_{n}$ is the principal order ideal
of $\mu(\Lambda^+)$: All permutations that precede $\mu(\Lambda^+)$. This defines the {\bf face subcomplex} $\mathcal{K}(\cT)\subset \Delta_\star$ whose faces are the ones all of whose vertices are in $\mu(\cS^{\cT})$, i.e., they precede $\mu(\Lambda^+)$.
In other words, faces correspond to intervals $[P,PP_0]$ with $PP_0\leq \mu(\Lambda^+)$.

\begin{corollary}\label{cor:order-ideal}
The face subcomplex $\mathcal{K}(\cT)\subset \Delta_\star$ equals $\mu(\Delta)$. In particular, the positive chamber $\Delta$ and the intersection of the mirrors of $\cM(\Delta)$ are contractible.
\end{corollary}
\begin{proof}
 By Corollary \ref{cor:convex-embedding} $C(\cT_\p)$ is a face of $\mathcal{P}_n$, and this is equivalent to  $\mu(\cT_\p)\cap \Delta_\star$ being a face of $\Delta_\star$ all whose vertices lie in $\mathcal{K}(\cT)$. Conversely, an interval  $[P,PP_0]$ corresponds to a partition $\p$ and a choice of values for each block. For $P$ the values are ordered increasingly within blocks. If $PP^0\leq  \mu(\Lambda^+)$, then all these permutations are the in  $\mu(\cT_\p)$ (where the choice of $\p\in \mathcal{P}$ depends also on the values on each block), and this implies that the ensuing face of $\mathcal{K}(\cT)$ lies in $\mu(\cT_\p)\cap \Delta_\star$.

The subcomplex  $\mathcal{K}(\cT)$ is contractible. One can use the nerve theorem
and transfer the problem to the geometric realization of the order complex of $\mathcal{K}(\cT)$ in  the dual fan of $\mathcal{P}_n$. There, the weak  Bruhat order is defined by passing through a root hyperplane from the negative to the positive side. The principal ideal of  $\mu(\Lambda^+)$ is the region cut out by all root hyperplanes for which the chamber corresponding to  $\mu(\Lambda^+)$ is in their negative side. This is a convex region, and therefore contractible. The argument for mirrors and their intersection is analogous. They define closed intervals  whose geometric realization is contractible, since it is star-shaped towards the region the minimum element. (Alternatively, one can invoke shellability of closed intervals \cite{BW}).
%
\end{proof}

\section{Hermitian matrices and split real semisimple Lie algebras}\label{sec:genenralizations}
We discuss some generalizations of the previous results.

\subsection{Hermitian matrices}\label{ssec:Hermitian}

 The tridiagonal isospectral set has the same definition. It is an arrangement of (complex) submanifolds of the flag manifold $\cO$ with combinatorial properties identical to those of the real case. The same Morse function as in the real case produces a $\mathrm{CW}$-complex decomposition of $\cT$ with only even dimensional cells; as many $2d$ dimensional cells as partitions of $\Lambda$ with $d$ ascents,
 \[\beta_{2d}(\cT)=\sum_{j=0}^{d+1} (-1)^j\begin{pmatrix}n+1\\j\end{pmatrix}\begin{pmatrix}n_1+d-j\\n_1\end{pmatrix}\cdots \begin{pmatrix}n_k+d-j\\n_k\end{pmatrix}.\]
 The tridiagonal isospectral set is related to the the parabolic Hessenberg variety
 \[\{[X]\in \Sl_\C/\mathbb{T}\,|\, X^{-1}\Lambda X\in \H\},\]
 where the equivalence classes are elements of the manifold of full flags. In the simple spectrum case the tridiagonal isospectral manifold and the regular Hessenberg variety are twin manifolds in the sense of  \cite{AB}. They may be non-homeomorphic, however, their $\mathbb{T}$-equivariant cohomology rings are canonically isomorphic. Such a twin relationship and their cohomological consequences is specific of the regular case. In spite of this, the tridiagonal isospectral set and the parabolic Hessenberg varieties seem to behave in an analogous manner.
 The singular locus of the Hessenberg variety has been described in combinatorial terms (for any complex semisimple Lie group using the corresponding Weyl group analysis) \cite{IP,PT}. A factorization of its Poincaré polynomial and some relations of its the Betti numbers with those of some irreducible components have been described \cite{PT}. It would be interesting to explore in detail these  analogies.

 Our results extend to quaternion matrices, where we regard a quaternion as complex $2\times 2$ matrix
 and $\sp_\mathbb{H}\subset \sl_\C$ as block $2\times 2$ complex matrices with a quaternion on each block.
 The cells of $\cT$ have dimension multiple of four and the resulting formula for its Betti numbers is the same.

 As it turns out, the constructions in $\sl_\C$ restricts to the constructions of the corresponding real forms $\sl_\R$ an $\sl_\mathbb{H}$, the fixed-point sets by the respective involutions
 \[X\mapsto \overline{X},\qquad X\mapsto -PJ\overline{X}JP^{-1},\]
 where \[J=\begin{pmatrix} 0& \mathrm{I}_n\\-\mathrm{I}_n & 0\end{pmatrix},\qquad P=(1,3,5,\dots,2n-1,2,4,\dots,2n).\]

\subsection{Split real semisimple Lie algebras}
The result in Sections 3 and 4 extend to tridiagonal isospectral sets for split real semisimple Lie algebras.

We work in the setting described in \cite[Sections 10-13]{Da2}. Let $\gg$ be a split real semisimple Lie algebra. Because we are interested in aspect of the adjoint action on $\gg$ we may choose any connected group integrating it. We let $\G$ be split real form of the simply connected group that integrates complexified Lie algebra $\gg_\C$, which is connected. We fix a root system for $\gg$ and a Chevalley basis for the corresponding root system in $\gg_\C$.

\begin{itemize}
 \item The Hessenberg space is the sum of the Cartan subalgebra and the $-1$ eigenspace (for the Cartan involution $\theta$) of sum of root spaces of simple roots and their opposites,
\[\ss_\H=\hh\oplus \sum_{\alpha\in \prod}\ss_i,\quad \ss_i=(\gg_{\alpha_i}\oplus \gg_{-\alpha_{i}})^\theta,\quad \prod=\{\alpha_1,\dots,\alpha_n\}.\]
The Chevalley basis defines a basis of $\ss_\H$ so that the summands above are coordinate subspaces.
The  elements of $\ss_\H$ which non-simple spectrum correspond to non-regular elements for the adjoint action on $\gg$.
\item The tridiagonal isospectral set is the intersection $\cT=\cO\cap \ss_\H$, where $\cO$ is the adjoint $\K$-orbit of $\Lambda\in \hh$. The non-regularity condition means that not every simple root is non-zero on $\Lambda$. The intersection $\cS^{\cT}=\cT\cap \hh$ is the orbit of $\Lambda$ by the Weyl group $\cS$, the right cosets in $\cS/\cS_{\cT}$, where $\cS_{\cT}$ is the centralizer of $\Lambda^+$, the  element of $\cT$ in the positive Weyl chamber $\Delta_\hh$.

\item  Subsets of $\prod$ correspond partitions in our sense. Such a subset  $\p$ defines a split reductive subalgebra $\gg(\p)$ by only keeping root spaces of roots that are sums of simple roots in  $\prod\backslash \p$. This is the centralizer of any $\Lambda\in \cS^\cT$ for which exactly the roots in $\p$ vanish. The group $\G(\p)$ is the centralizer of $\Lambda$. It Hessenberg space is
\[\H_\p=\hh\oplus \sum_{\alpha_i\in \prod\backslash\p}\ss_i.\]
A partition is simple at $\Lambda$ if no root which is sum of roots in $\prod\backslash \p$ vanishes at $\Lambda$,
\[\gg(\p)\cap \ll\subset  \ll^\Lambda,\qquad \ll=\sum_{\alpha>0}\gg_{-\alpha},\quad \ll^\Lambda=\sum_{\alpha>0,\,\alpha(\Lambda)\neq 0}\gg_{-\alpha}.\]
 Equivalently, $\Lambda$ is regular in the Hessenberg space of $\p$.

 Our notation deviates from Davis' in the sense that we use the complementary subset of roots. (We do it to work with usual partitions of an ordered set). In particular, our notation $\p\cap \p'$ corresponds to the intersection of the complementary subsets of roots.  Also, we do not work with the semisimple part $\gg(\p)$, which is a split real semisimple Lie algebra. The reason is that we consider the adjoint action of $\K(\p)$ on $\gg(\p)$, which factors through the center.
 \item The subgroup $\E$ is the centralizer in $\K$ of $\hh$, and it is isomorphic to ${(\mathbb{Z}^2)}^n$. One can choose generators
$\sigma_i$ that act on $\ss_\H$ by fixing $\hh$ and $\ss_{j}$, $j\neq i$, and by multiplying times -1  on  $\ss_i$. In the orbit type stratification for the action of $\E$ on $\ss_\H$ the centralizers
are subgroups $\E_\p\subset \E$, where $\p$ records the vanishing off-diagonal coordinates of $X\in \ss_\H$.
\item  If $X\in \cT$ belongs to $\cI$, then $\cT_\p$ is the intersection of its adjoint  $\K(\p)$-orbit with $\ss_\H$, where $\p$ is the partition that corresponds to the strata $\cI$ to which $X$ belongs. It is a manifold contained in $\H_\p$.
\item  For a partition $\p$ the subgroup $\cS_\p$ is generated by  the simple reflections of $\cS$ that correspond to roots in $\prod\backslash \p$. The positive chamber $\Delta_{\hh,\p}$  for the action of $\cS_\p$ on $\hh$ is the fundamental chamber containing the positive Weyl chamber $\Delta_{\hh}$.
That $\p$ be simple at $\Lambda$ means that it belongs to the interior of some chamber for the action of $\cS_\p$.
The simple partitions of $\cT$ are equivalence classes of partitions simple at some elements in $\cS^{\cT}$, where $\p(\Lambda)\sim \p(\Lambda')$ if $\Lambda$ and $\Lambda'$ are in the same orbit for the action of $\cS_\p$. For instance, for $\Lambda^+$ the only maximal partition simple at $\p$ is the one with $\cS_\p=\cS_\cT$.
\end{itemize}

With these preliminaries, that $\cT$ is an arrangement of tridiagonal isospectral manifolds  corresponding to maximal simple partitions of $\cT$ is proved as in Lemma \ref{lem:arrangement}.
If $\p_1,\dots,\p_r$ are simple partitions at $\Lambda$ and $\Lambda'$, the so $\p=\cap_i\p_i$ is. Therefore $\Lambda$ and $\Lambda'$ sit in the interior of chambers for the action of $\cS_{\p}$ on $\Delta_\hh$, so there is a unique element in  $\cS_{\p}$ taking $\Lambda$ to $\Lambda'$. This implies
\[\bigcap_i\cT_{\p_i}=\cT_{\p}.\]
A vector in the tangent space at $X\in \cT_\p$ is represented by a curve  in $\K(\p)$ at the identity, and the subgroups $\K(\p)$ and $\cap_i \K(\p_i)$ have the same Lie algebra.

The construction of adapted local coordinates has to be adjusted due to the non-linearity of $\L$. We recall that the choices of Cartan involution and positive roots determine an Iwasawa factorization of $\G$. The ensuing  $\mathrm{QR}$-type factorization is $\G=\K\U$, where $\U$ is the connected integration of the subalgebra $\uu=\hh\oplus\sum_{\alpha>0}\gg_\alpha$ (and $\L$ is the integration of $\ll$). The Toda vector field $\cY$ is defined by the same formula \eqref{eq:Toda}
\[X'=[X,\pi_{\kk}X],\quad \mathrm{I}=\pi_{\kk}+\pi_{\uu}.\]
Symes' factorization holds in this general setting and the map
\[\psi:\L\to \cO,\quad L\mapsto \kappa(L)^{-1}\Lambda \kappa(L),\quad \kappa:\G=\K\U\to  \K,\]
takes orbits of the adjoint action of $\mathrm{\exp(t\Lambda)}$ to the trajectories of the Toda vector field.

The chart uses the subspace of $\ll^\Lambda\subset \ll$ and the exponential-based map
\[ \ll^\Lambda=\sum_{\alpha>0,\,\alpha(\Lambda)\neq 0}\gg_{-\alpha},\qquad \psi_2:\ll^\Lambda\overset{\exp}{\longrightarrow} \L\overset{\psi}{\longrightarrow}\cO.\]
The subspace is the orthogonal complement to the centralizer of $\Lambda$ in $\ll$. (The root subspaces of the centralizer correspond to the repeated pairs; the innermost ones are those not in the image of the bracket of the centralizer with itself). Therefore   $\psi_2$ is a local diffeomorphism.
 The commutative triangle \eqref{cd:triangle} (c.f. \cite[Remark 6]{MT2}) shows that
\begin{equation}\label{eq:Hessenberg-image}\psi^{-1}_2(\cT)=\{L\in \ll^\Lambda\,|\, \exp(-L)\Lambda \exp(L)\in \ll^\Lambda_\H=\sum_{\alpha_i(\Lambda)\neq 0}\gg_{-\alpha_i}\}.\end{equation}
For any  $L=\sum L_\alpha \in \ll$, on each negative root space $\exp^{[-L,\Lambda]}$ is  a polynomial on the variables  $L_\alpha$ analogous to the right hand side of  \ref{eq:conj-diagonal}. For $L\in \ll^\Lambda$, using induction the height of root spaces we can write $L_{\alpha}$, $\gg_{-\alpha}\subset \ll^\Lambda$, as a monomial in the subdiagonal variables $L_i$, $\alpha_i(\Lambda)\neq 0$.
We define $\phi$ and $\Psi$ as in the proof of Proposition \ref{pro:adapted-coordinates},
and we have
\[\Psi^{-1}(\cT)=\T^\Lambda.\]
The reason is that because
\[L\in \gg(\p)\cap \H\Longrightarrow \exp^{[-L,\Lambda]}\in \gg(\p),\]
if we start with
\[L=\sum_{\alpha_i\in \prod\backslash \p} L_i,\]
 then we can  can add elements to $L$ in other negative root spaces using induction on the height  to obtain a solution to \eqref{eq:Hessenberg-image}. This means that $\psi_2^{-1}(\cT)$ is a graph over $\T^\Lambda$, and that
 \[\phi\circ \psi^{-1}_2(\cT)=\pr_\H\circ \psi_2^{-1}(\cT),\quad \pr_\H:\ll^\Lambda\to \ll^\Lambda_\H=\ll^\Lambda\cap \H.\]
 Because the orthogonal projection $\ll^\Lambda\to \ll^\Lambda_\H$ is equivariant with respect to the action by conjugation by $\exp(t\Lambda)$, the local diffeomorphism $\Psi:\T^\Lambda\to \cT$ sends $L'=[L,-\Lambda]$ to $\cY$. It is also equivariant with respect to the action of $\E$.

To discuss the topology of $\cT$, we fix $A\in \mathrm{int}(\Delta_\hh)$ and consider the linear function
\[f:\ss\to \R,\quad X\mapsto -\langle A,X\rangle. \]
This is a a Lyapunov Morse function for the opposite  Toda vector field $-\cY$ on $\cO$ \cite[Lemma 1 and Proposition 7]{Fa}.
The Hessian of $f$ in these local coordinates provided by $\Psi$  is minus \eqref{eq:Hessian}, where the differences of diagonal elements are replaced by evaluations over positive roots (see \cite[Proposition 1.4]{DKV}, where the formula is for the exponentiation from $\kk$ rather than $\ll$, and where our additional modifications do not change the Hessian since the local diffeomorphisms are tangent to the identity).
%

The restriction of the linear function $f:\cT\to \R$
is also a Lyapunov Morse function for the Toda vector field.
We define the auxiliary function $F_\Lambda:\ll_{\Lambda,\H}\to \R$, the auxiliary diagonal linear  vector field $-Y$ on $\ll^\Lambda_\H$ (replacing difference of eigenvalues by evaluations of simple roots), and the modification of $f$ to match the Hessian in adapted local coordinates, as we did in Section \ref{sec:homology} before Lemma \ref{lem:Milnor}. With the auxiliary data in place,  the construction of standard homotopy $h_\Lambda$ supported near $f^{-1}(\Lambda)\subset \cT$ in  Lemma \ref{lem:Milnor} also applies.

Lemma \ref{lem:runs} and Corollary \ref{cor:morse-ordering} also hold. The index of $F_\Lambda$ is the number of simple roots with $\alpha_i(\Lambda)>0$. The partition  that collects all simple roots with $\alpha_j(\Lambda)\leq 0$ is simple at $\Lambda$.  If $\p$ is a maximal simple partition at $ \Lambda$ preceded by it, then  $\T_\p$ contains the stable subspace of $F_\Lambda$. The absence of local maxima and minima is proved in a similar way.

The  $\mathrm{CW}$-complex structure on $\cT$ built out of $f$ has attaching maps with trivial degree. The proof using the fundamental reflections of $\E$ is identical.

Theorem \ref{thm:betti} holds true.
The element $\Lambda\in \cS^\cT\cong \cS/\cS_{\p}$ is identified with the minimal length representatives of the coset. For a simple root that,
 $\alpha_i(\Lambda)>0$ is equivalent to the corresponding reflection in $\cS$ being  a ascent for $\Lambda \in \cS$. Therefore the identification of the $d$-th Betti number with elements in $\cS^\cT\subset \cS$  with $d$ ascents is valid.

The tridiagonal isospectral set $\cT$ is aspherical. The proof via algebraic topology is exactly the same, since the result of Davis \cite{Da2} holds for tridiagonal isospectral manifolds of $\gg$. In particular $\widetilde{\cT}$  is a collection of euclidean spaces whose finite intersections are either empty or another euclidean space.

The mirror structure is defined in the same way:
The positive chamber $\Delta$ are the points whose coordinates are positive in the basis of $\ss_\H$; the mirrors are the collection of connected components of the fixed-point set $\Delta^{\sigma_i}$ over the fundamental reflections $\sigma_i$. The mirrors $\Delta^{\sigma}_j$ are in correspondence with the orbits of the action of $\cS_{\p_{\sigma}}$ on $\cS^\cT$. To the mirror we associate its intersection with $\hh$. To an $\cS_{\p_{\sigma}}$-orbit there we associate $\cT_{\p_{\sigma}}$ defined as its  $\K(\p_{\sigma})$-orbit intersected with $\ss_\H$. This is a tridiagonal isospectral set for $\gg(\p_{\sigma})$, and its positive chamber for the action of $\E_{\p_{\sigma}}$, which is connected, is the given mirror. (Should there be more than one $\cS_{\p_\sigma}$ orbit in the intersection, the corresponding tridiagonal isospectral sets would be disjoint, and thus the mirror would be disconnected).
If a non-empty intersection of $\cS_{\p_{\sigma_i}}$-orbits contains $\Lambda$ and $\Lambda'$, then any  element in $\cS$ of minimal length taking $\Lambda$ to $\Lambda'$ does not contain fundamental reflections corresponding to the simple roots $\alpha_i$, and therefore it belongs to $\cS_\p$, $\p=\cap_i\p_{\sigma_i}$. In terms of the corresponding tridiagonal isospectral sets,
\begin{equation}\label{eq:face-intervals}\bigcap_i\cT_{\p_{\sigma_i}}=\cT_{\p},
 \end{equation}
and the intersection of the mirrors is the positive chamber of $\cT_{\p}$ for the action of $\E_{\p}$. Because the modification of the standard homotopy in Proposition \ref{pro:contractible} can be done in the same way, the positive chamber, the mirrors, and their non-empty intersections are contractible.

Equation \eqref{eq:face-intervals} is just a geometric translation about intersections of intervals in $\cS^\cT$ coming face intervals in $\cS$ for the (right) weak Bruhat order. (Those that corresponds to the faces of the permutahedron for $\gg$ defined as the convex hull of a regular $\cS$-orbit in $\hh$). Equivalently, it says that  the non-empty intersection of  $\cS_{\p_i}$-orbits is a $\cS_{\p}$-orbit.
The conclusion also holds if one just assumes non-empty pairwise intersections. The argument is similar to the one in the proof of Theorem \ref{thm:mirror}.  We assume that the simple roots of the (possibly disconnected) Dynkin diagram labeled with the usual Bourbaki convention, starting from a short root at maximal distance form the fork root.
We order the partitions $\p_1,\dots,\p_s$ according to the root we removed. There is a splitting $\cS_{\p_i}=\cS_{\p_{i,1}}\times \cS_{\p_{i,2}}$, where the factors are generated by the fundamental reflections of $\cS$ associated to simple roots before and after $\alpha_i$, respectively. (If we remove a boundary root then one factor becomes trivial).
Let $\Lambda_1$ be in the $\cS_{\p_1}$-orbit determined by $\Delta_j^{\sigma_1}$.
There is a permutation $P_1\in \cS_{\p_1}$ taking $\Lambda_1$ to $\Lambda_2$ in the $\cS_{\p_2}$-orbit determined by $\Delta_j^{\sigma_2}$,
 unique up the the action of $\cS_{\p_1\cap \p_2}$. Next, there is a permutation  $P_2\in \cS_{\p_2}$ taking $\Lambda_2$ to $\Lambda_3$ in the $\cS_{\p_3}$-orbit determined by $\Delta_j^{\sigma_3}$, unique up to the action of
$\cS_{\p_2\cap \p_3}$. Because $\cS_{\p_{2,1}}\subset \cS_{\p_2\cap \p_3}$ one can assume that $P_2$ is the identity on this factor. This implies that $P_2\in \cS_{\p_1}$, and thus
\[\Lambda_3\in \Delta_j^{\sigma_1}\cap \Delta_j^{\sigma_2}\cap \Delta_j^{\sigma_3}.\]  An inductive argument proves that the intersection is non-empty.  Therefore \ref{thm:mirror} also holds, and this shows that $\cT$ is aspherical by a different method.

The construction of the embedding $\mu:\cT\to \cT_\star$ also holds. The family of subspaces $\cH$ is defined using the restriction of the Killing form on each stratum $\cT$.
The subset $\widehat{\cI}$ is the union of the $\cT_X$ as $X$ ranges over $\cI$. The map
\[\lambda:\widehat{\cI}\to \Delta_{\hh,\pp}\subset\hh\]
is defined by taking $X\in \cI$ and  intersecting $\cT_X$ with the positive chamber $\Delta_{\hh,\pp}$, where $\p$ is the partition corresponding to the stratum $\cI$. It is a submersion and $\cH$ is an Ehreshmann connection for $\lambda$. The corresponding parallel transport map
\[\mu:\cT_\p\to \cT_\star,\quad \cT_\star\cap \Delta_{\hh,\pp}\neq\emptyset\]
is an embedding equivariant with respect to the actions of $\E_\p$ and $\cS_{\p}$. The compatibility with the Toda vector field follows from the transverse intersection of the tridiagonal isospectral manifold with $\ss_\H$ \cite[Lemma 5]{MT}.
The compatibility for the different submanifold components of $\cT$ also holds. One takes the segment  from $\Lambda^+\in \cT\cap \Delta_\hh$ to $\Lambda_\star\in \Delta_\hh$. For each $\p\in \mathfrak{M}$, one makes parallel transport over the image the segment of the shortest element $P_\p$ in the coset of $\cS/\cS_{\p}$ that takes $\Lambda^+$ to $\Lambda^+_\p$. Restricting the parallel transport to $\cT_{\p\cap \p'}$  agrees with the parallel transport
$\mu_{\p\cap \p'}$ over the image by the segment of the only element $P_2$ of $\cS_{\p}$ taking $\cT_\p\cap \Delta_{\hh,\pp}$ to $\cT_{\p\cap \p'}\cap \Delta_{\hh,\p\cap \p'}$. (The product $P_2P_\p$ is the shortest element in the coset that takes $\Lambda^+$ to $\Lambda^+_{\p\cap \p'}$).
This allows to prove the contractibility of the intersection of the mirrors as in Corollary \ref{cor:order-ideal}. In particular,  Theorem \ref{thm:mirror-embedding} for  tridiagonal isospectral sets of split real semisimple Lie algebras.

 The permutahedron $\cP_n\subset \hh$ for the tridiagonal isospectral manifold $\cT_n\subset \ss_\H$ is the convex hull in $\hh$ of $\cT_n\cap \hh$. The face subcomplex $\mathcal{K}(\cT)\subset \cP_n$ is the collection of faces of $\cP_n$ all whose vertices lie in $\mu(\cS^\cT)$. Corollary \ref{cor:order-ideal} holds  using the moment map based homeomorphism $\Delta_n\longrightarrow \cP_n\subset \hh$ defined in \cite[Section 4.1]{BFR}.

\begin{remark}\label{rem:David}
As we already discussed, in \cite[Section 7]{Da2} Davis defines for the tridiagonal isospectral manifold $\cT_n$ a standard submanifold $\cT_{n,j}^{\E_\p}$ as a connected component of the fixed-point set of $\E_\p\subset\E$.  He shows that
\[1\to \pi_1(\cT_{n,j}^{\E_\p})\to \pi_1(\cT_n),\]
and that the the mirror structure of $\mathfrak{M}(\Delta_n)$ induces a mirror structure on $\cT_n^{\E_\p}$.
From a combinatorial perspective, standard submanifolds of $\cT_n$ correspond to faces of the permutahedron/face intervals of $\cS$.

The same result holds for connected components of the fixed-point set for the action of $\E_\p$ on $\cT$, that we could call standard parabolic subarrangements. They correspond to the mirrors and the non-empty intersections of mirrors of $\cM(\Delta)$, or, equivalently, the intersection with $\cS^\cT$ of face intervals of $\cS$. Finally,
the embedding $\mu:\cT\to \cT_n$ exhibits the latter subarrangements inside $\cT_n$, thus enlarging Davis' family of  standard submanifolds to one whose members have the same properties.
\end{remark}

\end{document}